\documentclass[11pt]{article}

\usepackage{enumitem} % For customizing enumerate styles

\pdfoutput=1

\usepackage{amsmath} %
\usepackage{parskip}
\usepackage{amsthm,amsfonts,amssymb,commath,bm,bbm, upgreek}
\usepackage{enumitem}
\usepackage{subcaption}
\usepackage{array}
\usepackage{mathtools}
\usepackage[affil-it]{authblk}
\usepackage[top=1 in,bottom=1in, left=1 in, right=1 in]{geometry}
\usepackage{hyperref}
\usepackage{cleveref}
\usepackage[backend=biber]{biblatex}
\usepackage[dvipsnames]{xcolor}

\allowdisplaybreaks

\makeatletter
\newcommand*{\rom}[1]{\expandafter\@slowromancap\romannumeral #1@}
\makeatother

\newcounter{relctr} %
\everydisplay\expandafter{\the\everydisplay\setcounter{relctr}{0}} %

\AtBeginDocument{} %

\newtheorem{theorem}{Theorem}
\newtheorem{lemma}{Lemma}
\newtheorem{proposition}{Proposition}
\newtheorem{definition}{Definition}

\newcommand\E{\mathbb{E}}

\newcommand{\betahh}{\hat{\hat{\beta}}}

\newcommand{\betaM}{\hat{\beta}^{M}}
\newcommand{\betaL}{\hat{\beta}^{L}}

\newcommand{\calN}{\mathcal{N}}
\newcommand{\R}{\mathbb{R}}

\newcounter{cnstcnt}

\newcounter{deltacnt}

\newcommand\simiid{\overset{i.i.d}{\sim}}

\newcommand\supp{\operatorname{supp}}

\DeclareMathOperator*{\argmin}{arg\,min}
\DeclareMathOperator*{\Exp}{Exp}

\usepackage{stackengine}

\def\square{\ifmmode\sqr\else{$\sqr$}\fi}
\def\sqr{\vcenter{
         \hrule height.1mm
         \hbox{\vrule width.1mm height2.2mm\kern2.18mm
\vrule width.1mm}
         \hrule height.1mm}}

\usepackage{color}

\begin{document}
\title{A note on the minimax risk of sparse linear regression}

\author[1]{Yilin Guo}
\author[1]{Shubhangi Ghosh}
\author[3]{Haolei Weng}
\author[1]{Arian Maleki}
\affil[1]{Columbia University}
\affil[2]{Michigan State University}

\date{\vspace{-5ex}}
\maketitle

\begin{abstract}Sparse linear regression is one of the classical and extensively studied problems in high-dimensional statistics and compressed sensing. Despite the substantial body of literature dedicated to this problem, the precise determination of its minimax risk remains elusive. This paper aims to fill this gap by deriving asymptotically constant-sharp characterization for the minimax risk of sparse linear regression. More specifically, the paper focuses on scenarios where the sparsity level, denoted as $k$, satisfies the condition $(k\log (p/k))/n \rightarrow 0$, with $p$ and $n$ representing the number of features and observations respectively. We establish that the minimax risk under isotropic Gaussian random design is asymptotically equal to $2 \sigma^2 k/n \log (p/k)$, where $\sigma$ denotes the standard deviation of the noise. In addition to this result, we will summarize the existing results in the literature, and mention some of the fundamental problems that have still remained open. 
\end{abstract}

%%%%%%%%%%%%%%%%%%%%%%%%%%%%%%%%%%%%%%%%%%%%%%
%% Please use \tableofcontents for articles %%
%% with 50 pages and more                   %%
%%%%%%%%%%%%%%%%%%%%%%%%%%%%%%%%%%%%%%%%%%%%%%
%\tableofcontents

\section{Sparse linear regression and minimaxity}

Consider the linear regression model
\begin{equation}\label{model::gaussian-model}
    y_{i} = x_{i}^{T}\beta + z_{i}, \qquad i=1, \ldots, n,
\end{equation}
in which $y_{i}\in\R $ denotes the response, $x_{i}\in \R^{p}$ represents the feature or covariate vector, $\beta\in \R^{p}$ is the unknown signal vector to be estimated, and finally $z_{1}, \ldots, z_{n} \overset{i.i.d.}{\sim} \mathcal{N}(0,\sigma^2)$ are normal errors. We are interested in studying this problem for a broad range of $p$, considering $p$ comparable with $n$, or even much larger than $n$. To ease one of the major concerns that linear regression procedures remain inconsistent unless $p/n\rightarrow 0$, following the rich literature of sparse linear regression \cite{hastie2009elements, buhlmann2011statistics, hastie2015statistical, wainwright2019high, fan2020statistical}, we consider the sparsity structure of the signal in this paper.  Specifically, we assume that the true regression coefficients are $k$-sparse:
\begin{equation} \label{param::sparse}
  \beta \in  \Theta_k:= \{ \beta \in \R^{p}: \|\beta\|_{0} \leq k \},
\end{equation}
where $\|\beta\|_{0}$ denotes the number of non-zero components of $\beta$. In evaluating the performance of estimators, the minimax framework has been one of the most popular approaches, aiming to obtain an optimal estimator which has the best worst-case performance among all estimators. In order to define the minimax risk, the first step is to consider a model for the design matrix $X\in \mathbb{R}^{n\times p}$. Two models have been considered in the literature for matrix $X:$

\begin{itemize}
\item Fixed design: In this model, matrix $X$ is regarded as a fixed entity, and no probabilistic assumptions are imposed upon it. Under this framework, the minimax risk is defined as
\begin{equation}\label{eq:minimxFixed}
R_F(X, \Theta_k, \sigma) := \inf_{\hat{\beta}} \sup_{ \beta \in \Theta_k} \E_{\beta} \|\hat{\beta} - \beta\|_2^{2},
\end{equation}
where $\|\cdot\|_2$ is the Euclidean norm. Note that in the above expression the expectation is with respect to the noise vector $z=(z_1,\ldots, z_n)$ only. 

\item Random design: In this model, $X$ is presumed to be generated via a known probabilistic mechanism. One common random design model posits that the rows of matrix $X$ are independent and identically distributed from a Gaussian distribution. In this case, the minimax risk is defined as 
\begin{equation}\label{eq::sparse-minimax}
    R_R(\Theta_k, \sigma) := \inf_{\hat{\beta}} \sup_{ \beta \in \Theta_k} \E_{\beta} \|\hat{\beta} - \beta\|_2^{2},
\end{equation}
where the expectation is with respect to both $X$ and $z$. 
\end{itemize}

Many researchers have considered characterizing the above two minimax risks for the sparse linear regression problem. However, obtaining the exact minimax risk is mathematically challenging and has remained open. Hence, researchers have explored approaches that aim to approximate the minimax risk. Below we summarize the existing approaches in the literature. On a related note, the minimax risk under other types of loss functions such as $\ell_q$-norm loss ($q\geq 1$) \cite{donoho1994minimax, ye2010rate, bellec2018slope}, prediction loss \cite{raskutti2011minimax, verzelen2012minimax, dalalyan2017prediction} and Hamming loss \cite{ji2012ups, butucea2018variable, ndaoud2020optimal}, has been also studied in the literature. The current paper is focused on the minimax risk with the squared loss as defined in \eqref{eq:minimxFixed} and \eqref{eq::sparse-minimax}.

\section{Approximation of minimax risk}
In this section, we would like to discuss three major approaches that researchers have explored for approximating the minimax risk.

\subsection{Rate-optimal results under fixed design}\label{ssec:fixeddesignrate}
 Given the complexity involved in precisely calculating the minimax risk, numerous studies have tried to obtain the relationship between $R_F(X, \Theta_k, \sigma)$ and parameters such as $k$, $n$, and $p$. In this approach, the aim is often to obtain a function of $k$, $n$, and $p$, call it $f_F(n,p,k)$, for which there exist two constants $c$ and $C$ such that 
\begin{equation}\label{eq:ro_minimax}
c f_F(n,p,k) \sigma^2 \leq R_F(X, \Theta_k, \sigma) \leq C f_F(n,p,k) \sigma^2. 
\end{equation}
As evident, the constant in the lower and upper bounds may be different. Under this setting, once $f_F(n,p,k)$ is characterized, an estimator $\hat{\beta}$ is called minimax rate-optimal if 
\[
\sup_{ \beta \in \Theta_k}\mathbb{E}_{\beta} \|\hat{\beta}-\beta\|_2^2 \leq \tilde{C}f_F(n,p,k)\sigma^2.
\]
Again, $\tilde{C}$ can be different from $C$ that appeared in the upper bound of the minimax risk. One important line of research on sparse linear regression, has been devoted to characterizing $f_F(n,p,k)$ and designing minimax rate-optimal estimators.  

The minimax risk $R_F(X, \Theta_k, \sigma)$ generally depends on the design matrix $X$. For instance, if a few columns of $X$ are linearly dependent, estimating the true $\beta$ becomes impossible and $R_F(X, \Theta_k, \sigma)$ equals infinite. Hence, certain aspects of $X$ need to appear in the upper bound (and lower bound). Some important conditions on $X$ have been proposed in the literature, including restricted isometry property \cite{candes2005decoding},  compatibility condition \cite{van2009conditions} and restricted eigenvalue (RE) condition \cite{bickel2009simultaneous}. We adopt a slightly stronger version of the RE condition from \cite{bellec2018slope} for later discussion.

\begin{definition} The matrix $X$ is said to satisfy the $SRE(k,c_0)$ condition for a given $c_0>0$ if its $j$th column is normalized $\|X_j\|_2 \leq \sqrt{n}$ for $j=1,\ldots, p$, and
\begin{align}
\label{sre:condef}
\theta(k, c_0) := \min_{\delta \in \mathcal{C}_{SRE} (k, c_0)\setminus \{0\}} \frac{\|X \delta\|_2}{\sqrt{n}\|\delta\|_2} >0,
\end{align}
where $\mathcal{C}_{SRE}(k, c_0) : = \big\{ \delta \in \mathbb{R}^p: \|\delta\|_1 \leq (1+ c_0) \sqrt{k} \|\delta\|_2 \big\}$ is a cone in $\mathbb{R}^p$
\end{definition}

According to this definition, if any $k$ columns of the matrix $X$ are linearly dependent (consequently the linear model is not identifiable on $\Theta_k$), then $\theta(k,c_0)=0$. Hence, the condition $\theta(k,c_0)>0$ ensures model identifiability. Moreover, it essentially requires that the Hessian for quadratic cost function has a positive curvature in directions over the cone $\mathcal{C}_{SRE}(k, c_0)$, which has implications for well-controlled estimation errors. To see how such a condition can be used to develop an upper bound for the minimax risk, consider the Lasso estimator \cite{tibshirani1996regression}:
\[
\hat{\beta}^{L} \in \argmin_{b \in \mathbb{R}^p} \frac{1}{2n}\|y-Xb\|_2^2 + \lambda \|b\|_1,
\]
where $y=(y_1,\ldots,y_n)$ is the response vector and $\|\cdot\|_1$ is the $\ell_1$-norm. The following theorem taken from \cite{bellec2018slope} obtains an upper bound for $\mathbb{E}_{\beta}\|\betaL- \beta\|_2^2$ based on the SRE condition.

\begin{theorem}[Corollary 4.4 in \cite{bellec2018slope}]
\label{thm:upperLASSOFixed}
Assume that $X$ satisfies the $SRE(k,7)$ condition. Let $\betaL$ denote the Lasso estimator with $\lambda$ satisfying 
\[
\lambda \geq {(8+2\sqrt{2}) \sigma} \sqrt{\frac{\log (2ep/k)}{n}}.
\]
Then, 
\begin{align}
\label{lasso:upper:fix}
R_F(X, \Theta_k, \sigma)\leq \sup_{\beta\in \Theta_k}\E_{\beta} \|\hat{\beta}^L- \beta\|_2^2 \leq  \frac{2401k\lambda^2}{64} \left( \frac{1}{\theta^4(k,7)} + \frac{1}{(\log(2ep))^2} \right). 
\end{align}

\end{theorem}

Based on our earlier discussion of $\theta(k,7)$ as a curvature-type condition, it is natural to see that $\theta(k,7)$ appears in the denominator of the upper bound. Following \cite{bellec2018slope}, we mention two significant features of this upper bound:
\begin{itemize}
\item Choosing $\lambda={(8+2\sqrt{2}) \sigma} \sqrt{\frac{\log (2ep)}{n}}$, the upper bound can be simplified to $C_{\theta}\frac{ \sigma^2 k\log p}{n}$ for some constant $C_{\theta}>0$ depending on $\theta(k,7)$ whenever $p\geq 2$. Although upper bounds of this form are known from past work for $\|\hat{\beta}^L-\beta\|_2^2$ (with high probability) \cite{bickel2009simultaneous, ye2010rate, negahban2012unified}, Theorem \ref{thm:upperLASSOFixed} is the first result to obtain such bounds in expectation, i.e. for $\E_{\beta} \|\hat{\beta}^L-\beta\|_2^2$. The previous work provided only bounds in probability with the confidence level tied to the tuning parameter $\lambda$, thus not allowing for control of the moments of $\|\betaL- \beta\|_2^2$. Similar bounds in probability have been obtained for several other estimators such as Dantzig selector \cite{bickel2009simultaneous} and square-root Lasso \cite{belloni2011square}.
\item Choosing $\lambda={(8+2\sqrt{2}) \sigma} \sqrt{\frac{\log (2ep/k)}{n}}$, the upper bound is improved to $\tilde{C}_{\theta}\frac{ \sigma^2k\log(p/k)}{n}$ (with a different constant $\tilde{C}_{\theta}$). In light of a minimax lower bound to be shown shortly, the rate $\frac{ k\log(p/k)}{n}$ is the minimax optimal rate. The fact that Lasso can achieve the optimal rate (not just the suboptimal rate $\frac{ k\log p}{n}$) was not known before the work \cite{bellec2018slope}\footnote{The optimal rate was known to be attained by $\ell_0$-constrained least squares under a sparse eigenvalue condition \cite{raskutti2011minimax}, although the bound was derived in probability instead of in expectation.}.
\end{itemize}

To evaluate the tightness of the upper bound in Theorem \ref{thm:upperLASSOFixed}, we present one lower bound for the minimax risk taken from \cite{verzelen2012minimax}.

\begin{theorem}[Proposition 6.2 in \cite{verzelen2012minimax}]
\label{thm:LOWERCandes}
Suppose each column of $X$ has been normalized to $\sqrt{n}$, i.e. $\|X_j\|_2 = \sqrt{n}$ for all $j=1,\ldots, p$. For any $k\leq (n/4) \wedge (p/2)$, it holds that
\begin{equation}\label{eq:minimaxlower}
R_F(X, \Theta_k, \sigma) \geq C_1 \sigma^2 \max\Big(\frac{k\log(ep/k)}{n}, \frac{\exp\{C_2k/n\log(p/k)\}}{n}\Big),
\end{equation}
where $C_1, C_2>0$ are two universal constants.
\end{theorem}

The type of lower bound $C_1\sigma^2k/n\log(p/k)$ has been derived in several papers under different conditions: \cite{raskutti2011minimax} assumes boundedness for both $k/n\log(p/k)$ and the maximum $2k$-sparse eigenvalue\footnote{The lower bound of \cite{raskutti2011minimax} is in probability, but it implies the lower bound in expectation via Markov's inequality.}; \cite{candes2013well} only requires $k\leq n$, and \cite{bellec2018slope} has the minimal condition $k \leq p/2$. An interesting aspect of Theorem \ref{thm:LOWERCandes} is that in the ultra-high dimensional regime where $k/n\log(p/k)\gg \log n$, the term $C_1\sigma^2\frac{1}{n}\exp\{C_2\frac{k}{n}\log(\frac{p}{k})\}$ becomes dominating in the lower bound. This does not contradict with the upper bound $\tilde{C}_{\theta}\sigma^2 k/n\log(p/k)$ from Theorem \ref{thm:upperLASSOFixed}, because the SRE condition will not hold in such an ultra-high dimensional setting. In contrast, when $k/n\log(p/k)$ is bounded, the SRE condition holds with high probability if the rows of $X$ are independent realizations of a large class of distributions \cite{rudelson2013reconstruction, lecue2017sparse, bellec2018slope}. 

In light of the aforementioned results, let us focus on the regime where $k/n\log(p/k)\leq C_3$ and $p\geq 2k$. Then, combining Theorems \ref{thm:upperLASSOFixed} and \ref{thm:LOWERCandes} we can reach the following conclusion about the minimax risk $R_F(X, \Theta_k, \sigma)$:
\begin{equation}\label{eq:ro_minimax_final}
\frac{c \sigma^2k \log (p/k) }{n} \leq R_F(X, \Theta_k, \sigma) \leq \frac{C  \sigma^2 k\log (p/k)}{n},
\end{equation}
as long as $X$ satisfies the $SRE(k,7)$ condition (which is true for a large set of matrices). Moreover, the Lasso estimator is minimax rate-optimal. This approximation approach offers very general non-asymptotic bounds such as the ones that have appeared in \eqref{eq:ro_minimax_final} to approximate the minimax risk, which is order-wise accurate. However, since the focus is on the optimal rate, the constants that appear in the upper and lower bounds might not be sharp. Results with sharper constants have been developed, for example, by considering a smaller value of $\lambda$ compared to the one in Theorem \ref{thm:upperLASSOFixed}. Since these results often appear as high-probability bounds instead of expectation bounds, we defer the discussion to Section \ref{ssec:asymptot}. One notable result is Theorem 3.13 in \cite{bellec2021second}, which provides a constant-sharp upper bound for Lasso, in terms of expected prediction error. As will be discussed in detail in Section \ref{proof:lasso:sharp}, some analyses in the current paper are motivated by \cite{bellec2021second} in order to obtain constant-sharp results for the minimax risk.

\subsection{Rate-optimal results under the random design}

In numerous scenarios, a more appropriate assumption regarding the data is that $X$ is generated via a random mechanism. Consequently, assessing the risk solely over the specific dataset observed thus far is not desirable. Instead, we want our estimation procedure to generalize well to future samples it encounters. In such cases, the random-design minimax risk  defined below is a better criterion:
\begin{equation}\label{eq:minimaxriskRD_2}
    R_R(\Theta_k, \sigma) := \inf_{\hat{\beta}} \sup_{ \beta \in \Theta_k} \E_{\beta} \|\hat{\beta} - \beta\|_2^{2}.
\end{equation}
Compared to $R_F(X, \Theta_k, \sigma)$ discussed in the last section, the expectation in \eqref{eq:minimaxriskRD_2} is now with respect to both the noise vector $z$ and the design matrix $X$, assuming independence between $z$ and $X$. This seemingly straightforward change may further complicate the task of approximating the minimax risk. To quickly understand the issue, suppose $\max_{j}\|X_j\|_2\leq \sqrt{n}$. Then the expectation of the upper bound in \eqref{lasso:upper:fix} gives an upper bound for $R_R(\Theta_k, \sigma)$. This requires calculating the expectation of $\frac{1}{\theta^4 (k,7)}$. However, calculating this expectation and proving that it is finite, is not straightforward. \cite{verzelen2012minimax} managed to obtain a tight upper bound for $R_R(\Theta_k, \sigma)$ by analyzing the risk of a model selection procedure, and a matching lower bound was also derived. We summarize the results in the following theorem:  

\begin{theorem}[Proposition 6.4 of \cite{verzelen2012minimax}]
Suppose $\{x_i\}_{i=1}^n \overset{i.i.d}{\sim} \mathcal{N}(0,\Sigma)$ and $\Sigma$ has ones on the diagonal. In the regime where $k/n\log(p/k)\leq C_1$ and $p\geq 2k$, it holds that 
\[
C_2 \frac{\sigma^2 k\log(p/k) }{n}\leq R_R(\Theta_k, \sigma) \leq C_3 \frac{\sigma^2k\log (p/k) }{n\bar{\theta}_{2k}},
\]
where $C_1,C_2,C_3>0$ are universal constants, and $\bar{\theta}_{2k}=\min_{\delta \in \Theta_{2k}\setminus \{0\}}\frac{\delta'\Sigma \delta}{\delta'\delta}$ is the minimal $2k$-sparse eigenvalue of $\Sigma$. 
\end{theorem}

As long as the sparse eigenvalue $\bar{\theta}_{2k}$ is bounded away from zero (which holds for rather general $\Sigma$'s \cite{raskutti2010restricted}), the minimax risk $R_R(\Theta_k, \sigma)$ satisfies the same type of upper and lower bounds as in \eqref{eq:ro_minimax_final}. Similar to the case of the fixed design, there can be a gap between the constants that appear in the upper and lower bounds. This issue has led researchers to explore another approach that we will describe next.

\subsection{Asymptotic approximation of minimax risk } \label{ssec:asymptot}

The issue raised earlier regarding the loose constants in the upper and lower bounds of the minimax risk has been recognized and deliberated upon by many researchers. One proposed solution is to acquire a reliable approximation of the constants through asymptotic arguments. This approach was initially advocated by Donoho and Johnstone in the orthogonal design setting where $\frac{1}{n}X^TX=I_p$. In this simpler case of sparse linear regression, Donoho and Johnstone \cite{donoho1992maximum, donoho1994minimax, johnstone19} demonstrated that
\[
 R_F(X, \Theta_k, \sigma) = \frac{(2+o(1))\sigma^2k \log(p/k)}{n},
 \]
 as $p \rightarrow \infty$ and $k/p \rightarrow 0$.

Unfortunately, there has been limited exploration of this approach for the broader context of linear regression. A precise asymptotic approximation for $R_F(X, \Theta_k, \sigma)$ (under more general design) or $R_R(\Theta_k, \sigma)$ is still lacking. To our knowledge, there exist only several works that have obtained sharp constants \cite{su2016slope, feng2019sorted, ndaoud2020scaled, bellec2018noise, bellec2019first}. The work~\cite{su2016slope} studies the Sorted L-One Penalized Estimator (SLOPE) introduced in \cite{bogdan2015slope} in a closely related context to ours. For $\lambda_{1}\geq\lambda_{2}\geq \cdots\geq \lambda_{p}\geq 0$, the SLOPE estimator is defined as a solution of the minimization problem 
\begin{equation*}
    \hat{\beta}_{\rm SLOPE}\in \argmin_{b\in \R^{p}} \frac{1}{2n}\|y-Xb\|_2^{2} + \sum_{j=1}^p\lambda_{j}|b|_{(j)} ,
\end{equation*}
where $|b|_{(1)}\geq|b|_{(2)}\geq \cdots \geq |b|_{(p)}$ are the order statistics of $|b_{1}|,|b_{2}|,\ldots, |b_{p}|$. To show the optimality of SLOPE, \cite{su2016slope} has proved the following results.

\begin{theorem}[Theorem 1.2 \& 1.3 in \cite{su2016slope}]\label{thm:weijie}
    Assume model \eqref{model::gaussian-model} with random Gaussian design $\{x_{i}\}_{i=1}^{n} \simiid \calN(0, I_{p})$ and parameter space \eqref{param::sparse}. Suppose $k/p \rightarrow 0$ and $(k\log p)/n \rightarrow 0$. 
    \begin{itemize}
    \item[(i)] For any $\varepsilon > 0$, 
    \begin{equation*}
        \inf_{\hat{\beta}} \sup_{\beta \in \Theta_k} \mathbb{P}\bigg( \|\hat{\beta}- \beta\|_2^{2} > \frac{(2-2\varepsilon)\sigma^2k\log(p/k)}{n} \bigg) \rightarrow 1.
    \end{equation*}
    \item[(ii)] Fix $0<q<1$ and set $\lambda_i = \sigma(1+\varepsilon)n^{-1/2}\Phi^{-1} (1- iq/(2p))$, where $\Phi^{-1}$ is the quantile function of a standard normal and $\varepsilon \in (0,1)$. Then, the SLOPE estimator satisfies
    \begin{equation*}
        \sup_{\beta \in \Theta_k} \mathbb{P}\bigg(\|\hat{\beta}_{{\rm SLOPE}}- \beta\|_2^{2}> \frac{(2+6\varepsilon)\sigma^2k\log(p/k)}{n} \bigg) \rightarrow 0.
    \end{equation*}
    \end{itemize}
\end{theorem}

Corollary 4.3 in~\cite{bellec2018noise} and Theorem 4.1 in~\cite{bellec2019first} obtain similar high-probability upper bounds for the Lasso estimator. Theorem 5 in~\cite{ndaoud2020scaled} obtains the sharp high-probability bound for a variant of the Iterative Hard Thresholding (IHT) estimator~\cite{blumensath2009iterative}. \cite{feng2019sorted} develops a sharp high-probability upper bound for a large family of concave penalized estimators.

Intuitively speaking, Theorem \ref{thm:weijie}, along with the corresponding results from~\textcolor{black}{\cite{bellec2018noise, bellec2019first, feng2019sorted,ndaoud2020scaled}}, suggests the following for the minimax risk $R_R(\Theta_k, \sigma)$: as $k/p \rightarrow 0$ and $(k\log p)/n \rightarrow 0$, the minimax risk is approximately $2\sigma^2k/n \log (p/k)$, and that \textcolor{black}{estimators such as} SLOPE, Lasso, and IHT asymptotically achieve the minimax risk. However, these theorems do not exactly characterize the minimax risk. This is because the minimax risk $R_R(\Theta_k, \sigma)$ defined in \eqref{eq::sparse-minimax} is based on the expected squared loss $\mathbb{E} \|\hat{\beta}- \beta\|_2^2$, while these theorems characterize the high-probability events for the squared loss $\|\hat{\beta}-\beta\|_2^2$. It is sometimes challenging to convert high-probability bounds on the squared loss to sharp (specially if we want sharp constants) bounds for the risk. To be more precise, Markov's inequality gives 
\begin{align*}
  \inf_{\hat{\beta}} \sup_{\beta \in \Theta_k} \mathbb{P}\bigg( \|\hat{\beta}- \beta\|_2^{2} > \frac{(2-2\varepsilon)\sigma^2k\log(p/k)}{n} \bigg) \leq \frac{\inf_{\hat{\beta}} \sup_{\beta \in \Theta_k} \E_{\beta} \|\hat{\beta}- \beta\|_2^{2}}{(2-2\varepsilon)\sigma^2k/n\log(p/k)}.
\end{align*}
Based on Part (i) of Theorem \ref{thm:weijie}, letting $n\rightarrow \infty$ and then $\varepsilon \rightarrow 0+$ yields
\begin{align}
\label{minimax:lower:form}
\liminf_{n\rightarrow \infty} \frac{R_R(\Theta_k,\sigma)}{\sigma^2k/n\log(p/k)}  \geq 2. 
\end{align}
We thus have converted the high-probability lower bound from Theorem \ref{thm:weijie} to a sharp lower bound for the minimax risk. However, the upper bound in Part (ii) of Theorem \ref{thm:weijie}, and the corresponding results from~\cite{bellec2018noise, bellec2019first, feng2019sorted,ndaoud2020scaled} are not directly transferable to a sharp upper bound for the minimax risk.

In summary, for both the fixed design (non-orthogonal) and random design cases, the asymptotic analysis of minimax risk that achieves sharp constants has remained an open problem. The main contribution of this paper is to show that for the isotropic Gaussian design,
    \begin{equation*}
        R_R(\Theta_k, \sigma) = \frac{(2+o(1))\sigma^2k\log(p/k)}{n}.
    \end{equation*}
As will be clarified later in the paper, in order to achieve this goal we have to address several technical challenges. We believe that our solutions can help create a viable path for evaluating the asymptotic approximations of $R_R(\Theta_k, \sigma)$ and $R_F (X, \Theta_k, \sigma)$ in more general settings, and even for problems beyond the sparse linear regression.

\section{Our main contribution}

As we described in the previous section, despite an extensive body of work on the sparse linear regression, the asymptotically exact characterization of the minimax risk has remained largely open. The main contribution of this paper is the following constant-sharp calculation for the minimax risk under Gaussian random designs.

\begin{theorem}\label{thm::first-order-sparse-minimax}
    Assume model \eqref{model::gaussian-model} with random Gaussian design $\{x_{i}\}_{i=1}^{n} \simiid \calN(0, I_{p})$ and parameter space \eqref{param::sparse}. As $k/p \rightarrow 0$ and \textcolor{black}{$\big(k \log (p/k)\big)/n \rightarrow 0$}, the minimax risk defined in \eqref{eq::sparse-minimax} satisfies
    \begin{equation*}
        R_R(\Theta_k, \sigma) = \frac{(2+o(1))\sigma^2k\log(p/k)}{n}.
    \end{equation*}
\end{theorem}

In order to establish Theorem \ref{thm::first-order-sparse-minimax}, we have used some of the techniques developed in \cite{bellec2018slope} and \cite{bellec2021second}. However, proving Theorem \ref{thm::first-order-sparse-minimax} is not a straightforward application of the results presented in these two papers for the following reasons: 

\begin{itemize}
\item For certain matrices, e.g. when a few columns of matrix $X$ are linearly dependent, the minimax risk will be infinite. Let's call such matrices ``infinite-risk matrices". Given our probabilistic assumption on matrix $X$, the probability of infinite-risk matrices is zero. However, the main concern in the minimax risk calculation is the matrices that are in the ``vicinity" of infinite-risk matrices. For such matrices, the minimax risk is very large but still finite, and the closer they are to the infinite-risk matrices the risk is expected to be larger. On the other hand, the likelihood of being closer to infinite-risk matrices tends to be lower. Hence, in order to establish the sharp minimax risk characterization, one has to obtain sharp bounds on the probability of the vicinity of infinite-risk matrices and on the minimax risk for such matrices. As will be clarified in Section \ref{roadmap:sec:part}, this requires a delicate analysis. 

\item A key part of our minimax risk calculations is to obtain a sharp upper bound in expectation, conditioning on the design matrix that satisfies a RE-type condition, for the Lasso estimator under a carefully chosen tuning parameter value. Towards this goal, we will refine the arguments in \cite{bellec2021second} (which have given constant-sharp results for prediction error) and \cite{bellec2018slope} (which have provided rate-optimal results for estimation error), to achieve the constant-sharp upper bound for estimation error. We provide more detailed discussions
in Section \ref{roadmap:sec:part}.

\end{itemize}

\section{Roadmap of the proof of Theorem \ref{thm::first-order-sparse-minimax}}
\label{roadmap:sec:part}

We first introduce some notations used in this section. We use $\mathbbm{1}_{\mathcal{A}}$ to represent the indicator function of the set $\mathcal{A}$. For a given vector $v = (v_{1}, \ldots, v_{p})\in \mathbb{R}^{p}$, $\norm{v}_{q} = \left(\sum_{i=1}^{p} |v_{i}|^{q}\right)^{1/q}$ for $q \in (0,\infty)$, and $\supp(v)=\{1\leq i\leq p: v_i\neq 0\}$ denotes its support. We use $\{e_j\}_{j=1}^p$ to denote the natural basis in $\mathbb{R}^p$. For a matrix $X\in \mathbb{R}^{n\times p}$, $X_j$ represents its $j$th column and $X_S\in \mathbb{R}^{n\times |S|}$ is the submatrix consisting of columns indexed by $S\subseteq \{1,2,\ldots, p\}$. For two non-zero real sequences $\{a_n\}_{n=1}^{\infty}$ and $\{b_n\}_{n=1}^{\infty}$, we use $a_n = o(b_n)$ to represent $|a_n/b_n| \rightarrow 0$ as $n \rightarrow \infty$, and use $a_n = O(b_n)$ for $\sup_n|a_n/b_n| < \infty$. For $a\in \mathbb{R}, a_+=\max(0, a)$.

Given the lower bound result we have obtained in \eqref{minimax:lower:form}, to show $R_R(\Theta_k, \sigma)=(2+o(1))\sigma^2k/n\log(p/k)$, it remains to prove the upper bound
\begin{align}
\label{main:goal:proof}
\limsup_{n\rightarrow \infty} \frac{R_R(\Theta_k, \sigma)}{\sigma^2k/n\log(p/k)}\leq 2.
\end{align}

To prove the upper bound of the minimax risk, the main idea is to construct a good estimator and obtain a sharp upper bound for its maximum risk. Towards this goal, consider the following two estimators:
\begin{itemize}
\item Lasso: 
\begin{equation}\label{eq::lasso-estimator0}
    \betaL \in \argmin_{b\in\R^{p}}\frac{1}{2n}\|y-Xb\|_{2}^{2} + \lambda \|b\|_{1}.
\end{equation}
\item Maximum likelihood estimator (MLE):
\begin{equation}\label{eq::MLE0}
    \betaM \in \argmin_{b\in \Theta_k} \|y-Xb\|_{2}^{2}.
\end{equation}
\end{itemize}
Both estimators are known to achieve the minimax optimal rate (in probability or in expectation). To obtain constant-sharp upper bounds, we construct an aggregated estimator that combines Lasso and MLE, taking the form
\begin{equation}\label{eq::beta-double-hat-estimator}
    \hat{\hat{\beta}} := \betaL \mathbbm{1}_{\mathcal{A}} + \betaM \mathbbm{1}_{\mathcal{A}^{c}}.
\end{equation}
Here, the event $\mathcal{A} :=\mathcal{A}(\delta_{0}, c_0,k)$ is defined as
\begin{equation}\label{eq::event-A}
    \mathcal{A}(\delta_{0}, c_0,k) := \bigg\{ X\in\R^{n\times p}: ~\max_{j=1,\ldots, p} \|X_{j}\|_{2} \leq (1+\delta_{0})\sqrt{n}, ~ \theta(k,c_0) \geq 1-\delta_{0} \bigg\},
\end{equation}
where $\theta(k,c_0)$ was introduced in \eqref{sre:condef} and $\delta_0,c_0>0$ are constants that will be specified shortly. When the design matrix $X$ is ``well conditioned" (in the sense of $\mathcal{A}$, holding with high probability), our estimator $\hat{\hat{\beta}}$ uses Lasso which will be shown to attain the sharp constant. Otherwise, our estimator resorts to the MLE that only induces a negligible error on rare events. The Lasso estimator depends on the tuning parameter $\lambda>0$, and our choice of $\lambda$ in the proof will be 
\begin{align}
\label{lambda:set:up}
\lambda_{\varepsilon} := (1+\varepsilon)\sigma\sqrt{\frac{2\log(p/k)}{n}},
\end{align}
where $\varepsilon>0$ is an arbitrarily small constant. This choice of tuning parameter value is consistent with the existing works that have obtained sharp high-probability upper bounds (e.g. \cite{ndaoud2020scaled, bellec2018noise, bellec2019first}). For the given tuning parameter $\lambda_{\varepsilon}$ in \eqref{lambda:set:up}, we set the corresponding constants $\delta_0,c_0$ of \eqref{eq::event-A}:
\begin{align}
\label{cons:value}
\delta_0=\Big(1+\frac{\varepsilon}{2}\Big)^{\frac{1}{3}}-1, ~~c_0=8\sqrt{2}\varepsilon^{-1}\Big(1+\frac{\varepsilon}{2}\Big)^{\frac{2}{3}}+2\varepsilon^{-1}+2.
\end{align}

Our goal is to prove that for any fixed $\varepsilon\in (0,1)$, our estimator $\hat{\hat{\beta}}$ constructed in \eqref{eq::beta-double-hat-estimator} satisfies 
\begin{align}
\label{goal:est}
\sup_{\beta \in \Theta_k} \E \| \hat{\hat{\beta}} - \beta\|_2^2 \leq \frac{(2+o(1))f(\varepsilon)\sigma^2 k \log (p/k)}{n},
\end{align}
where $f(\varepsilon)>0$ and $f(\varepsilon)\rightarrow 1$ as $\varepsilon \rightarrow 0+$.
The above result implies the upper bound for the minimax risk,
\[
\limsup_{n\rightarrow \infty} \frac{R_R(\Theta_k, \sigma)}{\sigma^2k/n\log(p/k)}\leq 2f(\varepsilon), \quad \forall \varepsilon \in (0,1). 
\]
Further letting $\varepsilon \rightarrow 0+$ yields \eqref{main:goal:proof}, and this will finish the proof.

To prove \eqref{goal:est}, from the construction of $\betahh$ we have
\begin{eqnarray*} \label{eq::lasso-MLE-decomposition}
    \E\|\hat{\hat{\beta}} - \beta\|^{2}_2= \E \left(\|\betaL - \beta\|^{2}_2 \mathbbm{1}_{\mathcal{A}} \right) + \E \left( \|\betaM - \beta\|_2^{2} \mathbbm{1}_{\mathcal{A}^{c}} \right),
\end{eqnarray*}
and we aim to show 
\begin{align}
&\sup_{\beta \in \Theta_k} \E \left(\|\betaM - \beta\|_2^{2}\mathbbm{1}_{\mathcal{A}^{c}}\right) = o\Big(\sigma^2k/n\log (p/k)\Big), \label{mle:upper:key} \\
&\sup_{\beta \in \Theta_k} \E \left(\|\betaL - \beta\|_2^{2} \mathbbm{1}_{\mathcal{A}}\right) \leq (2+o(1))f(\varepsilon)\sigma^2k/n\log (p/k). \label{lasso:upper:key}
\end{align}

\subsection{Proof of \eqref{mle:upper:key}}

First note that from H\"{o}lder's inequality we have
\begin{eqnarray}\label{eq::cauchy-schwarz-expansion-for-MLE-bound}
\E \left(\|\betaM - \beta\|_2^{2}\mathbbm{1}_{\mathcal{A}^{c}}\right) \leq  \Big( \E \|\betaM - \beta\|_{2}^{m} \Big)^{\frac{2}{m}} \cdot \big(\mathbb{P}(\mathcal{A}^c)\big)^{\frac{m-2}{m}}, \quad m\in (2,\infty). 
\end{eqnarray}
Hence, we will prove that $ \sup_{\beta \in \Theta_k}\Big( \E \|\betaM - \beta\|_{2}^{m} \Big)^{\frac{2}{m}}=O(\sigma^2 k/n\log (p/k))$ and that  $\mathbb{P}(\mathcal{A}^c)= o(1)$. Let us start with the simpler one, i.e. $\mathbb{P}(\mathcal{A}^c)= o(1)$. The following lemma proves this claim.

\begin{lemma}\label{lem::gaussian-matrix-satisfy-SRE}
    Assume the design matrix $X\in\R^{n\times p}$ has i.i.d. $\calN(0, 1)$ entries. For any constants $c_{0}>0, \delta_0\in (0,1)$ and $k\in\{1,\ldots, p\}$, there exist absolute constants $C, C'>0$ such that if
    
    \begin{equation}\label{eq::large-n-SRE-condition}
        n\geq C\delta_0^{-2}(3+c_0)^2k\log (ep/k),
    \end{equation}
    then with probability at least $1-2\exp\big( -C'k\log (ep/k) \big)$ we have
    \begin{equation}\label{eq::random-matrix-SRE-condition}
        \max_{j=1,\ldots, p} \|X_{j}\|_2\leq (1+\delta_{0})\sqrt{n}, \quad \theta(k,c_0) \geq 1-\delta_{0}.
    \end{equation}
\end{lemma}

\begin{proof}
Similar results exist in the literature of high-dimensional statistics. For completeness, we provide a proof based on uniform matrix deviation inequality from Lemma \ref{matrix:dev}. Define the set
\[
T:=\Big\{\delta\in\mathbb{R}^p: \|\delta\|_1\leq (1+c_0)\sqrt{k}, \|\delta\|_2=1\Big\},
\]
and apply Lemma \ref{matrix:dev} to obtain: $\forall u>0$, with probability at least $1-2e^{-u^2}$ it holds that

\begin{align}
\label{start:eq:one}
\sup_{\delta\in T}\Big|\|X\delta\|_2-\sqrt{n}\Big|&\leq C_1\Big(\mathbb{E}\sup_{\delta\in T}|\delta^Th|+u\Big), \quad h\sim \mathcal{N}(0,I_p), 
\end{align}
where $C_1 > 0$ is an absolute constant. To obtain a sharp bound for the expectation above, define
\[
R=\Big\{\delta\in \mathbb{R}^p: \|\delta\|_0\leq k, \|\delta\|_2\leq 1 \Big\},
\]
and denote its convex hull by ${\rm conv}(R)$. We now show that $T\subseteq (2+c_0){\rm conv}(R)$. For any $\delta\in T$, let $G_1$ index the $k$ largest  (in absolute value) non-zero coefficients of $\delta$, $G_2$ index the next $k$ largest non-zero coefficients, and so on, leading to $\{1\leq j\leq p:\delta_j\neq 0\}=\cup_{j=1}^mG_j$. Further define
\[
\delta^{(j)}_i=\delta_i \mathbbm{1}_{i\in G_j}, ~~ i=1,\ldots, p, j=1,\ldots, m.
\]
It is clear that $\{\mathbf{0}\}\cup \{\delta^{(j)}/\|\delta^{(j)}\|_2\}_{j=1}^m\subseteq R$. Consider the following decomposition:
\[
\delta=(2+c_0)\cdot \Bigg(\frac{2+c_0-\sum_{j=1}^m\|\delta^{(j)}\|_2}{2+c_0}\cdot \mathbf{0}+\sum_{j=1}^m \frac{\|\delta^{(j)}\|_2}{2+c_0}\cdot \frac{\delta^{(j)}}{\|\delta^{(j)}\|_2}\Bigg).
\]
Then $T\subseteq (2+c_0){\rm conv}(R)$ holds whenever $\sum_{j=1}^m\|\delta^{(j)}\|_2\leq 2+c_0$. This inequality can be obtained as follows:
\begin{align*}
\sum_{j=1}^m\|\delta^{(j)}\|_2 &\leq 1+ \sum_{j=2}^m\|\delta^{(j)}\|_2 \leq 1+\sum_{j=1}^{m-1}\sqrt{k}\cdot \frac{\|\delta^{(j)}\|_1}{k} \\
&\leq 1+\frac{1}{\sqrt{k}}\|\delta\|_1 \leq 2+c_0,
\end{align*}
where the second inequality holds because by definition every non-zero component of $\delta^{(j)}$ is no larger in magnitude than the average non-zero component of $\delta^{(j-1)}$; the first and last inequalities are due to the simple fact that  $\delta\in T$. As a result, we can proceed to obtain
\begin{align}
\label{start:eq:two}
\mathbb{E}\sup_{\delta\in T}|\delta^Th|&=\mathbb{E}\sup_{\delta\in T}\delta^Th\leq \mathbb{E}\sup_{\delta\in (2+c_0){\rm conv}(R)}\delta^Th \nonumber \\
&=(2+c_0) \cdot \mathbb{E}\sup_{\delta\in {\rm conv}(R)}\delta^Th=(2+c_0) \cdot \mathbb{E}\sup_{\delta\in R}\delta^Th \nonumber \\
&\leq (2+c_0)C_2 \sqrt{k\log(ep/k)},
\end{align}
where $C_2>0$ is an absolute constant, and the last inequality can be found from Exercise 5.7 in \cite{wainwright2019high}. Combining \eqref{start:eq:one} and \eqref{start:eq:two} yields that with probability at least $1-2e^{-u^2}$,
\begin{align*}
\sup_{\delta\in T}\Big|\|X\delta\|_2-\sqrt{n}\Big|\leq C_1 \Big[(2+c_0)C_2\sqrt{k\log(ep/k)}+u\Big]
\end{align*}
Choosing $u=C_2\sqrt{k\log(ep/k)}$ gives
\[
\mathbb{P}\Big(1-\delta_0\leq \|X\delta\|_2/\sqrt{n}\leq 1+\delta_0,~\forall \delta\in T\Big)\geq 1-2e^{-C_2^2k\log(ep/k)},
\]
as long as $n\geq \delta_0^{-2}C_1^2C_2^2(3+c_0)^2k\log(ep/k)$.
 The proof is thus completed because 
\begin{align*}
\theta(k,c_0)=\inf_{\delta\in T}\|X\delta\|_2/\sqrt{n}, \quad \sup_{1\leq j\leq p}\|Xe_j\|_2/\sqrt{n}\leq \sup_{\delta\in T}\|X\delta\|_2/\sqrt{n}
\end{align*}
\end{proof}

As is clear from Lemma \ref{lem::gaussian-matrix-satisfy-SRE}, under the conditions $k/p\rightarrow 0, (k\log (p/k))/n\rightarrow 0$, the event $\mathcal{A}=\mathcal{A}(\delta_{0}, c_0,k)$ with constants $\delta_0,c_0$ in \eqref{cons:value} satisfies 

\[
\mathbb{P}(\mathcal{A}^c)\leq 2e^{-C'k\log (ep/k)}=o(1).
\]

The next step is to obtain an upper bound for $\Big( \E \|\betaM - \beta\|_{2}^{m} \Big)^{\frac{2}{m}}$. This is done in the next proposition. 

\begin{proposition}\label{prop::MLE-bound}
    Assume model \eqref{model::gaussian-model} with isotropic Gaussian design. Suppose $k/p\rightarrow 0$ and $\frac{k\log(p/k)}{n} \rightarrow 0$. Then,
    \begin{equation*}
        \sup_{\beta \in \Theta_k}\Big( \E \|\hat{\beta}^M-\beta\|_{2}^{m} \Big)^{\frac{2}{m}}=O(\sigma^2 k/n \log (p/k)), \quad \forall m \in [1,\infty).
    \end{equation*}
\end{proposition}

The proof of this proposition is long, hence is deferred to Section \ref{prop::MLE-bound:proof}.

\subsection{Proof of \eqref{lasso:upper:key}}
\label{proof:lasso:sharp}
In establishing this bound, the main challenge is that we would like to obtain an upper bound with the sharp constant $2$. We adapt the proof strategy in \cite{bellec2021second} and \cite{bellec2018slope} to first show that the estimation error of Lasso $\betaL$ can be accurately approximated by the error of the soft thresholding estimator under a sequence model, as stated in Proposition \ref{prop:lasso:denoiser}. Then we can calculate the estimation error of the latter estimator with the desirable sharp constant, in a relatively straightforward way. 
\textcolor{black}{
\begin{proposition}
\label{prop:lasso:denoiser}
Consider model \eqref{model::gaussian-model} with isotropic Gaussian design. Recall $\mathcal{A}=\mathcal{A}(\delta_{0}, c_0,k)$ in \eqref{eq::event-A} with constants $\delta_0,c_0$ from \eqref{cons:value}. Then, there exist constants $C_{\varepsilon}, c_{\varepsilon}>0$ only depending on $\varepsilon$ such that as long as $p/k>c_{\varepsilon}$, the Lasso estimator $\betaL$ in \eqref{eq::lasso-estimator0} with tuning parameter $\lambda_{\varepsilon}$ satisfies
\begin{align}
\label{approx:denoiser:risk}
&\sup_{\beta \in \Theta_k} \E \left(\|\betaL - \beta\|_2^{2} \mathbbm{1}_{\mathcal{A}}\right) \nonumber \\
\leq &~\frac{1}{(1-\delta_0)^4} \underbrace{\sup_{|S|\leq k}\Bigg\{\sum_{j\in S}\mathbb{E}(g_j-\lambda_{\varepsilon})^2\mathbbm{1}_{\mathcal{A}}+\sum_{j\in S^c}\mathbb{E}(|g_j|-\lambda_{\varepsilon})_+^2\mathbbm{1}_{\mathcal{A}}\Bigg\}}_{Q(\Theta_k)}\nonumber \\
&~+C_{\varepsilon} \frac{\sigma^2}{nk\log(p/k)},
\end{align}
where $g=\frac{1}{n}X^Tz$ and $S={\rm supp}(\beta)$.
\end{proposition}
}
\textcolor{black}{
The proof of this proposition is deferred to Section \ref{ssec:proof:prop5}. }

\textcolor{black}{
As is clear, the second term in the upper bound \eqref{approx:denoiser:risk} is negligible compared to the order $(k\log(p/k))/n$. The constant $\frac{1}{(1-\delta_0)^4}$ from the first term tends to one as $\varepsilon \rightarrow 0+$, so we expect the term $Q(\Theta_k)$ to give the correct order with the sharp constant 2. Before calculating $Q(\Theta_k)$, we recognize that $Q(\Theta_k)$ can be in fact viewed as the risk of the soft thresholding estimator under a sequence model. Specifically, consider the sparse sequence model: $\tilde{y}_j=\beta_j+g_j, j=1,\ldots, p$, and the soft thresholding estimator 
\[
\tilde{\beta}=\argmin_{\beta\in \mathbb{R}^p}\frac{1}{2}\sum_{j=1}^p(\tilde{y}_j-\beta_j)^2+\lambda_{\varepsilon}\sum_{j=1}^p|\beta_j| \Rightarrow  \tilde{\beta}_j={\rm sign}(\tilde{y}_j)(|\tilde{y}_j|-\lambda_{\varepsilon})_+.
\]
Then, conditioning on $X\in \mathcal{A}$, we obtain
\begin{align*}
\sup_{\beta\in \Theta_k}\mathbb{E}\|\tilde{\beta}-\beta\|_2^2 &=\sup_{|S|\leq k} \Bigg\{\sum_{j\in S}\lim_{\beta_j\rightarrow \infty}\mathbb{E}(\tilde{\beta}_j-\beta_j)^2+\sum_{j\in S^c}\mathbb{E}|\tilde{\beta}_j|^2\Bigg\}\\
&=\sup_{|S|\leq k} \Bigg\{\sum_{j\in S}\mathbb{E}(g_j-\lambda_{\varepsilon})^2+\sum_{j\in S^c}\mathbb{E}(|g_j|-\lambda_{\varepsilon})_+^2\Bigg\},
\end{align*}
where the first equality holds since the risk $\mathbb{E}(\tilde{\beta}_j-\beta_j)^2$ increases as $|\beta_j|$ increases. Therefore, $Q(\Theta_k)$, if conditioning on $X\in \mathcal{A}$, can be interpreted as the supremum risk of $\tilde{\beta}$.}

\textcolor{black}{
We now calculate $Q(\Theta_k)$ in detail. Note that $g_j|X \sim \mathcal{N}(0, \sigma^2n^{-2}\|X_j\|_2^2)$. We first calculate the conditional expectation with respect to $z$ and then calculate the expectation with respect to $X$, yielding
\begin{align}
\label{sharp:constant:compute}
&~Q(\Theta_k)\nonumber \\
\leq &~\sup_{|S|\leq k}\Bigg\{\sum_{j\in S}\mathbb{E}(\lambda_{\varepsilon}^2+\sigma^2n^{-2}\|X_j\|_2^2)\mathbbm{1}_{\mathcal{A}} \nonumber \\
&~+\sum_{j\in S^c}\mathbb{E}\sigma^2n^{-2}\|X_j\|_2^2\exp\bigg(\frac{-n^2\lambda_{\varepsilon}^2}{2\sigma^2\|X_j\|_2^2}\bigg)\mathbbm{1}_{\mathcal{A}}\Bigg\} \nonumber \\
\leq&~ k\Big(\lambda_{\varepsilon}^2+\frac{(1+\delta_0)^2\sigma^2}{n}\Big)+p\frac{(1+\delta_0)^2\sigma^2}{n}\exp\bigg(\frac{-n\lambda_{\varepsilon}^2}{2(1+\delta_0)^2\sigma^2}\bigg) \nonumber \\
=&\frac{2(1+\varepsilon)^2\sigma^2k\log(p/k)}{n}+\frac{k(1+\delta_0)^2\sigma^2}{n}+\frac{(1+\delta_0)^2\sigma^2p}{n}\cdot \Big(\frac{k}{p}\Big)^{\frac{(1+\varepsilon)^2}{(1+\varepsilon/2)^\frac{2}{3}}}.
\end{align}
Here, the first inequality has used Lemma G.1 in \cite{bellec2021second}, the second inequality holds since $\|X_j\|_2\leq (1+\delta_0)\sqrt{n}$ on the event $\mathcal{A}$, and the last equality used the choice of $\delta_0$ in \eqref{cons:value}. Finally, under the scaling $p/k\rightarrow \infty$, it is direct to verify that \eqref{approx:denoiser:risk}-\eqref{sharp:constant:compute} together lead to
\begin{align*}
\sup_{\beta \in \Theta_k} \E \left(\|\betaL - \beta\|_2^{2} \mathbbm{1}_{\mathcal{A}}\right) \leq \frac{(1+\varepsilon)^2}{[2-(1+\frac{\varepsilon}{2})^{\frac{1}{3}}]^4}\cdot \frac{(2+o(1))\sigma^2k\log(p/k)}{n}.
\end{align*}
This proves \eqref{lasso:upper:key}.}

\section{Conclusion}

 In this paper, we study the minimax risk of the sparse linear regression problem. Despite the considerable volume of research dedicated to this area, as discussed in the paper, the constant-sharp analysis of minimax risk is still rather limited in the literature. To contribute along this line, we explored the asymptotic scenario where $(k \log(p/k))/n \rightarrow 0$ and derived the sharp asymptotic minimax risk $2\sigma^2 k/n\log(p/k)$ under isotropic Gaussian design. Along the way, we provided a summary of existing literature results and highlighted some of the fundamental issues that have remained unresolved.

\appendix

\begin{center}
{\LARGE\bf Proofs of technical results}
\end{center}

We collect the notations used throughout the proof sections for convenience. For an integer $n$, let $[n]=\{1,2,\ldots, n\}$. We use $\mathbbm{1}_{\mathcal{A}}$ to represent the indicator function of the set $\mathcal{A}$. For a given vector $v = (v_{1}, \ldots, v_{p})\in \mathbb{R}^{p}$, $\norm{v}_{0} = \# \{i: v_{i}\neq 0\}$, $\norm{v}_{\infty} = \max_{i}|v_{i}|$, $\norm{v}_{q} = \left(\sum_{i=1}^{p} |v_{i}|^{q}\right)^{1/q}$ for $q \in (0,\infty)$, $\supp(v)=\{1\leq i\leq p: v_i\neq 0\}$ denotes its support, and $v_S\in \mathbb{R}^{|S|}$ denotes the subvector consisting of coordinates in $S\subseteq [p]$. The inner product of two vectors $a,b$ is written as $\langle a, b\rangle$. We use $\{e_j\}_{j=1}^p$ to denote the natural basis in $\mathbb{R}^p$. For a matrix $X\in \mathbb{R}^{n\times p}$, $\sigma_{\min}(X)$ denotes its smallest singular value and $\sigma_{\max}(X)$ (or $\|X\|_2$) denotes its largest singular value; $X_j$ represents its $j$th column and $X_S\in \mathbb{R}^{n\times |S|}$ is the submatrix consisting of columns indexed by $S\subseteq [p]$. The $p\times p$ identity matrix is denoted by $I_p$. For two real numbers $a$ and $b$, $a \vee b$ and $a \wedge b$ represent $\max(a, b)$ and $\min(a, b)$, respectively. For two non-zero real sequences $\{a_n\}_{n=1}^{\infty}$ and $\{b_n\}_{n=1}^{\infty}$, we use $a_n = o(b_n)$ to represent $|a_n/b_n| \rightarrow 0$ as $n \rightarrow \infty$, and use $a_n = O(b_n)$ for $\sup_n|a_n/b_n| < \infty$. For $a\in \mathbb{R}, a_+=\max(0, a)$. The notation $x\overset{d}{=}y$ means that the random variables $x$ and $y$ have the same distribution. For a random vector $x$, the notation $\|x\|_{\psi_2}$ denotes its sub-Gaussian norm. Finally, we reserve the notations $\Phi(y)$ and $\Phi^{-1}(y)$ for CDF of $\mathcal{N}(0,1)$ and its inverse function, respectively.

\section{Preliminaries}

\begin{lemma}[Binomial coefficient, Exercise 0.0.5 in \cite{vershynin2018high}]\label{lem:stirling}
For a given positive integer $p$, 
\[
{p \choose s} \leq \Big(\frac{ep}{s}\Big)^s,
\]
holds for all integers $s\in [1,p]$.
\end{lemma}

\begin{lemma}[Covering number of unit sphere, Corollary 4.2.13 in \cite{vershynin2018high}] \label{lem::covering-number}
    The covering numbers $\calN(S^{n-1}, \epsilon)$ of the unit Euclidean sphere $S^{n-1}:=\{v\in \mathbb{R}^n: \|v\|_2=1\}$ satisfy
    \begin{equation*}
        \calN(S^{n-1}, \epsilon) \leq \left(\frac{3}{\epsilon}\right)^{n},\quad \forall \epsilon \in (0,1].
    \end{equation*}
\end{lemma}

\begin{lemma}[Matrix deviation inequality, Exercise 9.1.8 in \cite{vershynin2018high}]
\label{matrix:dev}
Let $A$ be an $n\times p$ matrix whose rows $e_i^TA$ are independent, isotropic and sub-Gaussian random vectors in $\mathbb{R}^p$. Then for any given subset $T\subset \mathbb{R}^p$, the event 
\[
\sup_{x \in T}\Big|\|Ax\|_2-\sqrt{n}\|x\|_2\Big| \leq CK^2\Big(\gamma(T)+u \cdot {\rm rad}(T)\Big)
\]
holds with probability at least $1-2e^{-u^2}$. Here, $\gamma(T)=\mathbb{E}\sup_{x\in T}|h^Tx|, h\sim \mathcal{N}(0, I_p); {\rm rad}(T)=\sup_{x\in T}\|x\|_2; K=\max_{i}\|e_i^TA\|_{\psi_2}$; $C>0$ is a universal constant and $u \geq 0$ is any constant. 

\end{lemma}

\begin{lemma}[$\chi^{2}$-concentration, Lemma 2 of \cite{6283602}] \label{lem::chi-square-concentration}
    Let $g_1,\ldots, g_d \overset{i.i.d.}{\sim} \calN(0,1)$. Then,
    \begin{align*}
       &\mathbb{P}\Big( \sum_{i=1}^{d} g_{i}^{2} < d(1-\tau) \Big) \leq e^{\frac{d}{2} \big(\tau + \log(1-\tau)\big)}, \quad \forall \tau \in (0,1).
    \end{align*}
\end{lemma}

\begin{lemma}[Singular values of Gaussian matrices, Corollary 5.35 in \cite{vershynin2010introduction}]\label{lem::eval-conc}
   Let $A$ be an $N \times n$ ($N\geq n$) matrix whose entries are independent standard normal random variables. Then, 
    \begin{equation*}
        \mathbb{P}\Big(\sqrt{N} -\sqrt{n} - t \le \sigma_{\min}(A) \le \sigma_{\max} (A) \le \sqrt{N} + \sqrt{n} + t\Big) \ge 1 - 2e^{-\frac{t^2}{2}},~~\forall t>0.
    \end{equation*}
\end{lemma}

\begin{lemma}\label{lem::Gaussian-order-statistics}
    Let $g_{1}, \ldots, g_{p}\simiid \calN(0,1)$, and $|g|_{(1)}\geq |g|_{(2)}\geq \cdots \geq |g|_{(p)}$ denote the order statistics of $\big(|g_{1}|, \ldots, |g_{p}|\big)$.
    \begin{enumerate}[label=(\roman*)]
        \item For all $2\leq k \leq p$, $\E |g|_{(k)} \leq \sqrt{2\log (\frac{2p}{k-1})}$.
        \item For all $1\leq k \leq p$, $\mathbb{P}\Big( |g|_{(k)} - \E |g|_{(k)} \geq u \Big) \leq e^{-\frac{u^2}{2}}, ~\forall u >0$.       
    \end{enumerate}
\end{lemma}

\begin{proof}

  Prove $(i)$.  Let $U_1,\ldots, U_p \overset{i.i.d.}{\sim}{\rm Unif}(0,1)$ with order statistics $U_{(1)}\leq U_{(2)}\leq \cdots \leq U_{(p)}$, and $Y_{1},\ldots, Y_{p}\simiid \Exp(1)$ with order statistics $Y_{(1)} \geq Y_{(2)} \geq \ldots \geq Y_{(p)}$. We will use the following well known distributional results: for $k=1,2,\ldots, p$,
\begin{align*}
|g|_{(k)}\overset{d}{=} \Phi^{-1}(1-U_{(k)}/2), \quad Y_{(k)}\overset{d}{=} \log\frac{1}{U_{(k)}}, \quad Y_{(k)}-Y_{(k+1)}\overset{d}=\Exp(k),
\end{align*}
where $\Phi^{-1}(\cdot)$ is the inverse function of CDF of standard normal and $Y_{(p+1)}:=0$. Then, 
     \begin{align*}
        (\E |g|_{(k)})^2 &\leq \E |g|_{(k)}^{2} = \E \big( \Phi^{-1}(1-U_{(k)}/2) \big)^2 \overset{(a)}{\leq} \E \big[ 2 \log \frac{2}{U_{(k)}} \big] \\
        &= 2\log 2 + 2\sum_{j=k}^{p}\mathbb{E}(Y_{(j)}-Y_{(j+1)})=2\log 2 + 2\sum_{j=k}^{p}\frac{1}{j} \overset{(b)}{\leq} 2 \log \frac{2p}{k-1}.
    \end{align*}
   Here, $(a)$ is due to the Gaussian tail bound $1-\Phi(t)\leq e^{-\frac{t^2}{2}}, \forall t>0$; $(b)$ holds because $\sum_{j=k}^{p}\frac{1}{j}\leq \sum_{j=k}^p\int_{j-1}^{j}\frac{1}{x} d x = \log p - \log (k-1)$ for $k\geq 2$.

    Prove $(ii)$. Given that each order statistic is a 1-Lipschitz function (e.g. Example 2.29 in \cite{wainwright2019high}), we have
    \[
    \big||g|_{(k)}-|\tilde{g}|_{(k)}\big|\leq \sqrt{\sum_{i=1}^p(|g_i|-|\tilde{g}_i|)^2} \leq \sqrt{\sum_{i=1}^p(g_i-\tilde{g}_i)^2},
    \]
    where in the second inequality we have used $\big||a|-|b|\big|\leq |a-b|, \forall a, b\in \mathbb{R}$. Hence, $|g|_{(k)}$, as a function of $(g_1,\ldots, g_p)$, is 1-Lipschitz as well. Applying standard Gaussian concentration inequality (Theorem 5.6 in \cite{blm13}) completes the proof. 
\end{proof}

\begin{lemma}[Proposition E.1 in \cite{bellec2018slope}]\label{lem::prop-E.1}
    Let $g_{1}, \ldots, g_{p}\simiid \calN(0,1)$, and $|g|_{(1)}\geq |g|_{(2)}\geq \cdots \geq |g|_{(p)}$ denote the order statistics of $\big(|g_{1}|, \ldots, |g_{p}|\big)$. Then for any $s\in \{1\ldots,p\}$ and all $t>0$, we have
    \begin{equation*}
        \mathbb{P}\bigg( \frac{1}{s} \sum_{j=1}^{s} |g|_{(j)}^{2} > t \log (2p/s) \bigg) \leq (2p/s)^{1-\frac{3t}{8}}.
    \end{equation*}
\end{lemma}

\begin{lemma}\label{lem::median-inequality}
    Under the assumptions of Lemma \ref{lem::prop-E.1}, then for any fixed $\delta_1>0$, 
    \begin{align*}
        & \mathbb{P} \bigg( \Big( \max_{1\leq j\leq k} \frac{|g|_{(j)}}{4\sqrt{\log(2p/j)}} \Big) \vee \frac{|g|_{(k+1)}}{(1+\delta_1)\sqrt{2\log(p/k)}} \leq 1 \bigg) \\
        & \geq 1 - \frac{k}{2p} -\exp\Big\{-\frac{1}{2}\Big((1+\delta_1)\sqrt{2\log(p/k)}-\sqrt{2\log(2p/k)}\Big)^2\Big\}>\frac{1}{2},
    \end{align*}
when $p/k$ is large enough.
\end{lemma}

\begin{proof}
The proof follows that of Proposition E.2 in \cite{bellec2018slope}. 

Lemma \ref{lem::prop-E.1} with $t=16/3$ and the inequality $|g|_{(j)}^{2} \leq \frac{1}{j}\sum_{l=1}^{j}|g|_{(l)}^{2}$ imply 
    \begin{equation}
       \mathbb{P} \bigg( |g|_{(j)}^{2} \leq \frac{16}{3} \log (2p/j) \bigg) \geq 1 -\frac{j}{2p}, \quad j=1,\ldots, p. \label{eq::gaussian-order-stat-16/3-concent}
    \end{equation}
    Let $q\geq 0$ be an integer such that $2^{q} \leq k < 2^{q+1}$. Applying \eqref{eq::gaussian-order-stat-16/3-concent} to $j=2^{l}$ for $l=0, \ldots, q-1$ and using the union bound, we obtain that the event
    \begin{equation*}
        \Omega_{0} := \bigg\{ \max_{l=0,\ldots, q-1} \frac{|g|_{(2^{l})}\sqrt{3}}{4\sqrt{\log(2p/2^{l})}} \leq 1 \bigg\}
    \end{equation*}
    satisfies $\mathbb{P}(\Omega_{0}) \geq 1 - \sum_{l=0}^{q-1} \frac{2^{l}}{2p} = 1 - \frac{2^{q}-1}{2p} \geq 1 - \frac{k}{2p}$. For any $j< 2^{q}$, there exists $l\in\{0, \ldots, q-1\}$ such that $2^{l}\leq j < 2^{l+1}$. On the event $\Omega_{0}$,
    \begin{equation*}
        |g|_{(j)} \leq |g|_{(2^{l})} \leq \frac{4}{\sqrt{3}} \sqrt{\log \frac{2p}{2^{l}}} \leq \frac{4}{\sqrt{3}} \sqrt{\log \frac{4p}{j}} \leq 4 \sqrt{\log \frac{2p}{j}}, \quad \forall j <2^{q}.
    \end{equation*}
    And for $2^{q}\leq j \leq k$,
    \begin{equation*}
        |g|_{(j)} \leq |g|_{(2^{q-1})} \leq \frac{4}{\sqrt{3}} \sqrt{\log \frac{2p}{2^{q-1}}} < \frac{4}{\sqrt{3}} \sqrt{\log \frac{8p}{j}} \leq 4 \sqrt{\log \frac{2p}{j}}.
    \end{equation*}
    Thus, on the event $\Omega_{0}$ we have $|g|_{(j)} \leq 4 \sqrt{\log (2p/j)}$ for all $j=1,\ldots, k$. 

    In addition, using Lemma \ref{lem::Gaussian-order-statistics}, we have
    \begin{equation*}
        \mathbb{P}\bigg( \frac{|g|_{(k+1)}}{(1+\delta_1)\sqrt{2\log \frac{p}{k}}} \geq 1 \bigg) \leq e^{-\frac{1}{2}\big((1+\delta_1)\sqrt{2\log(p/k)}-\sqrt{2\log(2p/k)}\big)^2}.
    \end{equation*}
\end{proof}

\section{Proof of Propositions \ref{prop::MLE-bound}-\ref{prop:lasso:denoiser}}

\subsection{Proof of Proposition \ref{prop::MLE-bound}}\label{prop::MLE-bound:proof}

Recalling the definition of $\betaM$ in \eqref{eq::MLE0}, we start with the basic inequality
\[
\|y- X \betaM\|_2^2 \leq \|y- X \beta\|_2^2.
\]
With $y=X\beta+z$, this implies
\begin{equation}\label{eq:fub:eq:MLE}
     \frac{1}{n}\|X(\betaM-\beta)\|_{2}^{2} \leq \frac{2}{n}z^{T} X(\betaM-\beta).
\end{equation}
For a given $s\in\{1,\ldots,p\}$, define
\begin{equation}\label{eq::v-term-def}
    V_{s}:= \inf_{\Delta\in T_s} \frac{1}{n}\|X \Delta\|_{2}^{2}, \quad T_s:=\Big\{\Delta\in \mathbb{R}^p: ~ \|\Delta\|_{2}=1,~\|\Delta\|_{0}\leq s \Big\}.
\end{equation}
Both $\betaM$ and $\beta$ are in $\Theta_k$, hence $\|\betaM-\beta\|_0 \leq 2k$. We then continue from \eqref{eq:fub:eq:MLE} to obtain 
\begin{align*}
    V_{2k}\cdot\|\betaM-\beta\|_{2}^{2} &\leq \frac{1}{n}\|X(\betaM - \beta)\|_{2}^{2} \leq \frac{2}{n} z^{T} X(\betaM-\beta)\\
    &\leq \frac{2}{\sqrt{n}} \|\betaM-\beta\|_{2} \cdot \sup_{u\in T_{2k}} z^{T}Xu/\sqrt{n}.
\end{align*}
Therefore, 
\begin{equation*}
   \|\betaM - \beta\|_{2} \leq \frac{2}{\sqrt{n}V_{2k}} \cdot\sup_{u\in T_{2k}} z^{T}Xu/\sqrt{n}.
\end{equation*}
Then, using Cauchy–Schwarz inequality, we have
\begin{equation}\label{eq:upperbound2:MLE}
    \sup_{\beta\in \Theta_k}\E \|\betaM - \beta\|_{2}^{m} \leq  \frac{2^{m}}{n^{\frac{m}{2}}} \cdot \bigg(\E \frac{1}{V_{2k}^{2m}}\bigg)^{\frac{1}{2}} \cdot \bigg( \E \bigg(\sup_{u\in T_{2k}} z^{T} Xu/\sqrt{n} \bigg)^{2m}\bigg)^{\frac{1}{2}}.
\end{equation}
Hence, we need to bound the two terms on the right-hand side of \eqref{eq:upperbound2:MLE}. This is done in Lemmas \ref{lem::inverse-min-eigenvalue} and \ref{lem::gaussian-complexity}. Combining \eqref{eq:upperbound2:MLE} with Lemmas \ref{lem::inverse-min-eigenvalue} and \ref{lem::gaussian-complexity} completes the proof of Proposition \ref{prop::MLE-bound}. 

    \begin{lemma}\label{lem::inverse-min-eigenvalue}
    Suppose the matrix $X\in\R^{n\times p}$ has i.i.d $\mathcal{N}(0,1)$ entries. For $s\in \{1,\ldots,p\}$, let $V_{s}$ be defined as in \eqref{eq::v-term-def}. If $(s\log (ep/s))/n\rightarrow 0$, then, for every fixed $r>0$, we have
    \begin{equation*}
        \E \frac{1}{V_{s}^{r}} = O(1).
    \end{equation*}
\end{lemma}

\begin{lemma}\label{lem::gaussian-complexity}
Suppose the matrix $X\in\R^{n\times p}$ has i.i.d $\mathcal{N}(0,1)$ entries, and is independent of $z\sim \mathcal{N}(0, \sigma^2I_n)$. For $s\in \{1,\ldots,p\}$, let $T_{s}$ be defined as in \eqref{eq::v-term-def}. Then,
    \begin{equation*}
    \E \Big( \sup_{u\in T_s} z^{T}Xu/\sqrt{n} \Big)^{q} \leq c_q \Big(\sigma \sqrt{s \log (ep/s)}\Big)^q,  \quad \forall q \in [1,\infty),
    \end{equation*}
    for some constant $c_q>0$ that only depends on $q$.
\end{lemma}

\begin{proof}[Proof of Lemma \ref{lem::inverse-min-eigenvalue}]

Throughout the proof, we fix $s\in\{1,\ldots,p\}$ and let $V:=V_{s}$ for notational simplicity. We have
\begin{equation}\label{eq::inverse-min-eigenvalue-expectation}
    \E \frac{1}{V^{r}} = \E \left(\frac{1}{V^{r}} \mathbbm{1}_{(V\leq x)} \right) + \E
    \left(\frac{1}{V^{r}} \mathbbm{1}_{(V>x)}\right),
\end{equation}
where we set\footnote{Many other choices of $x$ will work as well. We do not aim to optimize the constant.} $x= {\rm e}^{-8}$. It is clear that
\begin{equation}\label{eq:2ndterm:Vbound}
\E \left(\frac{1}{V^{r}} \mathbbm{1}_{(V>x)}\right)< \frac{1}{x^{r}}.
\end{equation}
Hence, in the rest of the proof, we aim to obtain an upper bound for $\E \left( \frac{1}{V^{r}} \mathbbm{1}_{(V\leq x)}\right)$. Towards this goal, we first bound $\mathbb{P}(V \leq 1-t)$. For $\forall t \in (0,1)$, using the union bound, we have
    \begin{align}
        \mathbb{P}(V\leq 1-t) &= \mathbb{P}\Big(\min_{S\subseteq [p]:|S|=s} \inf_{\|\Delta\|_{2}=1} \frac{1}{n}\|X_{S}\Delta\|_{2}^{2} \leq 1-t \Big)  \nonumber\\
        &\leq \binom{p}{s} \cdot \max_{S\subseteq [p]:|S|=s}\mathbb{P}\Big( \inf_{\|\Delta\|_{2}=1} \frac{1}{n}\|X_{S}\Delta\|_{2}^{2} \leq 1-t \Big). \label{V:major:eq}
    \end{align}
    We focus on $\mathbb{P}\Big( \inf_{\|\Delta\|_{2}=1} \frac{1}{n}\|X_{S}\Delta\|_{2}^{2} \leq 1-t \Big)$ for now. This bound needs to be sharp for small values of $1-t$ to help us bound $\mathbb{E} (\frac{1}{V^r} \mathbbm{1} (V \leq x))$.\footnote{Note that standard concentration bounds for the singular values of Gaussian matrices are not sharp enough. } Define the set
    \[
    \mathcal{S}^{s-1} := \{ \Delta \in \mathbb{R}^s : \|\Delta\|_2=1 \}.
    \]
   We discretize the set $\mathcal{S}^{s-1}$ using an $\varepsilon$-net  and write the union bound over the net in the following way. Let $\mathcal{N}(\varepsilon)$ denote the $\varepsilon$-net of $\mathcal{S}^{s-1}$. Then for $\forall \Delta \in \mathcal{S}^{s-1}$, there exists a $\Delta' \in \mathcal{N}(\varepsilon)$ such that $\|\Delta - \Delta'\|_{2} \leq \varepsilon$ and
    \begin{align}\label{eq::epsilon-net-ineq}
        &~\|X_{S}\Delta\|_{2}^{2} \nonumber \\
        =&~ \|X_{S}\Delta'\|_{2}^{2} + \langle X_{S}(\Delta - \Delta'), X_{S}(\Delta+\Delta') \rangle \nonumber \\ 
        \geq&~\inf_{\Delta \in \mathcal{N}(\varepsilon)} \|X_{S}\Delta\|_{2}^{2} - \|\Delta- \Delta'\|_2 \|\Delta+ \Delta'\|_2 \Big\langle X_{S} \frac{(\Delta - \Delta')}{\|\Delta-\Delta'\|_2}, X_{S}\frac{(\Delta+\Delta')}{\|\Delta + \Delta'\|_2} \Big\rangle \nonumber \\   
        \geq &~ \inf_{\Delta \in \mathcal{N}(\varepsilon)} \|X_{S}\Delta\|_{2}^{2} - 2\varepsilon\sigma^2_{\max}(X_S),
    \end{align}
    where to obtain the last inequality, we have used the fact that $\|\Delta-\Delta'\|_2 \leq \varepsilon, \|\Delta+ \Delta'\|_2 \leq 2$ and the Cauchy-Schwartz inequality. Define the event
    \[
    \mathcal{D}:= \Big\{\frac{1}{n}\sigma^2_{\max}(X_S) \leq \frac{1}{1-t} \Big\},
    \]
    and let $2\varepsilon = (1-t)^{2}$. We use \eqref{eq::epsilon-net-ineq} to have
    \begin{align}\label{eq:upper:for:lowest}
         &\mathbb{P}\Big( \inf_{\Delta \in \mathcal{S}^{s-1}} \frac{1}{n}\|X_{S}\Delta\|_{2}^{2} \leq 1-t \Big)\nonumber \\
        \leq&~ \mathbb{P}\Big( \inf_{\Delta \in \mathcal{N}(\varepsilon)} \frac{1}{n}\|X_{S}\Delta\|_{2}^{2} \leq 2(1-t),~\mathcal{D} \Big) + \mathbb{P}(\mathcal{D}^{c}) \nonumber \\
        \leq&~ \frac{6^{s}}{(1-t)^{2s}}\cdot \max_{\Delta\in \mathcal{N}(\varepsilon)} \mathbb{P} \Big(\frac{1}{n}\|X_{S}\Delta\|_{2}^{2} \leq 2(1-t) \Big) + \mathbb{P}(\mathcal{D}^{c}),
    \end{align}
    where the last inequality uses the union bound and the result $|\mathcal{N}(\varepsilon)| \leq (3/\varepsilon)^{s}$ from Lemma \ref{lem::covering-number}. Our next step is to bound the following two quantities from \eqref{eq:upper:for:lowest}: 

\begin{itemize}

\item $\mathbb{P}\Big(\frac{1}{n}\|X_{S}\Delta\|_{2}^{2} \leq 2(1-t) \Big)$: Since $S$ is a fixed set and $\Delta$ is a fixed unit-norm vector, we know $\|X_{S}\Delta\|_{2}^{2} \sim \chi_{n}^{2}$. Applying Lemma \ref{lem::chi-square-concentration} gives $\forall t \in (1/2, 1)$,
\begin{equation}\label{p:xdelta:lower1}
\mathbb{P}\Big(\frac{1}{n}\|X_{S}\Delta\|_{2}^{2} \leq 2-2t\Big) \leq \exp\Big[ \frac{n}{2}\Big( 2t -1 + \log(2-2t) \Big) \Big].
\end{equation}

\item $\mathbb{P}(\mathcal{D}^{c})$: A direct use of Lemma \ref{lem::eval-conc} yields: $\forall  t\in (1-(\sqrt{s/n}+1)^{-2}, 1)$,
\begin{align}
\label{D:com:eq}
\mathbb{P}(\mathcal{D}^{c})&=\mathbb{P}\Big(\sigma_{\max}(X_S)\geq \sqrt{\frac{n}{1-t}}\Big) \nonumber \\
&\leq 2\exp\Big(-\frac{1}{2}\Big(((1-t)^{-1/2}-1)\sqrt{n}-\sqrt{s}\Big)^2\Big).
\end{align}

\end{itemize}
    
We now use the bounds \eqref{V:major:eq} and \eqref{eq:upper:for:lowest}-\eqref{D:com:eq} to obtain an upper bound for $\E \left(\frac{1}{V^{r}} \mathbbm{1}_{(V\leq x)} \right)$. First, the $r$th moment of $\frac{1}{V} \mathbbm{1}_{(V\leq x)}$ can be obtained via its tails:
\begin{align*}
&\E \left(\frac{1}{V^{r}} \mathbbm{1}_{(V\leq x)} \right)=\int_0^{\infty}ru^{r-1}\mathbb{P}\Big(\frac{1}{V} \mathbbm{1}_{(V\leq x)}> u\Big)du \\
=&~\int_0^{1/x}ru^{r-1}\mathbb{P}\Big(\frac{1}{V} \mathbbm{1}_{(V\leq x)}> u\Big)du+ \int_{1/x}^{\infty}ru^{r-1}\mathbb{P}\Big(\frac{1}{V} \mathbbm{1}_{(V\leq x)}> u\Big)du \\
=&~x^{-r}\mathbb{P}\big(V\leq x\big)+\int_{1-x}^1r(1-t)^{-r-1}\mathbb{P}(V< 1-t)dt,
\end{align*}
where we use a change of variable $t=1-u^{-1}$ in the last equality. Plugging the bounds of \eqref{V:major:eq}, \eqref{eq:upper:for:lowest}-\eqref{D:com:eq} into the above integral, we have
\begin{align}
\E \left(\frac{1}{V^{r}} \mathbbm{1}_{(V\leq x)} \right)\leq &~x^{-r}\mathbb{P}\big(V\leq x\big)\nonumber\\
&~+ r\binom{p}{s}6^s\int_{1-x}^1(1-t)^{-r-1-2s}e^{\frac{n}{2}\big( 2t -1 + \log(2-2t) \big)}dt \nonumber \\
&~+ 2r\binom{p}{s}\int_{1-x}^1(1-t)^{-r-1}e^{-\frac{1}{2}\big(((1-t)^{-1/2}-1)\sqrt{n}-\sqrt{s}\big)^2}dt \nonumber \\
:=&~x^{-r}\mathbb{P}\big(V\leq x\big)+I_1+I_2. \label{three:terms}
\end{align}
Note that \eqref{p:xdelta:lower1} and \eqref{D:com:eq} can be applied here, because $x<\frac{1}{2}\wedge (\sqrt{s/n}+1)^{-2}$ for $x=e^{-8}$ and $s\leq n$.

For the term $I_1$, we can bound as follows: 
\begin{align}
    I_{1} &= r\binom{p}{s}6^s \int_{1-x}^1(1-t)^{\frac{n}{2}-r-1-2s}e^{\frac{n}{2}\big(2t -1 +\log2\big)} dt \nonumber\\
    & \leq r\binom{p}{s} 6^{s}  e^{\frac{n}{2}(1+\log2)} \int_{1-x}^{1} (1-t)^{\frac{n}{2}-r-1-2s} dt \nonumber\\
    & = r\binom{p}{s} 6^{s} e^{\frac{n}{2}(1+\log 2)} \frac{x^{\frac{n}{2}-r-2s}}{\frac{n}{2}-r-2s} \nonumber \\
    & \leq \frac{r}{\frac{n}{2}-r-2s}\exp\Big(s\log(6ep/s)-(3.5-0.5\log2)n+8r+16s\Big)=o(1),
\label{eq::I_1-integration}
  \end{align}
where we used Lemma \ref{lem:stirling} in the last inequality, and the last equality can be easily verified under the scaling condition $(s\log (ep/s))/n\rightarrow 0$.

Regarding the term $I_2$, we have
\begin{align}
\label{I2:bound}
    I_{2} = &~2r\binom{p}{s}\int_{1-x}^1\exp\Big((r+1)\log\frac{1}{1-t}-\frac{1}{2}\big(((1-t)^{-1/2}-1)\sqrt{n}-\sqrt{s}\big)^2\Big)dt \nonumber \\
    \leq &~2r\binom{p}{s}\int_{1-x}^1\exp\Big(\frac{-n/8+r+1}{1-t}\Big)dt\leq 2r\binom{p}{s}e^{-e^8(n/8-r-1)-8} \nonumber \\
    \leq &~2r\exp\Big(s\log(ep/s)-e^8(n/8-r-1)-8\Big)=o(1).
\end{align}
Here, the first inequality uses the fact that when $n\geq s$ and $0\leq1-t\leq x=e^{-8}$, it holds that $\big(((1-t)^{-1/2}-1)\sqrt{n}-\sqrt{s}\big)^2\geq \frac{n}{4(1-t)}$ and $\log\frac{1}{1-t}\leq \frac{1}{1-t}$; the second inequality holds by replacing $t$ with $1-x$ in the integrand; the third inequality uses Lemma \ref{lem:stirling}; and the last equality is seen under the scaling condition $(s\log (ep/s))/n\rightarrow 0$.

Putting together \eqref{three:terms}-\eqref{I2:bound} completes the proof. 

\end{proof}

\begin{proof}[Proof of Lemma \ref{lem::gaussian-complexity}]

Given that $X$ and $z/\sigma$ have independent $\mathcal{N}(0,1)$ entries, with a conditioning (on $z$) argument, we can obtain
\begin{align}
\label{basic:form:eq}
    \E \Big( \sup_{u\in T_s} z^{T}Xu/\sqrt{n} \Big)^{q} =\mathbb{E}\Big(\frac{\|z\|_2}{\sqrt{n}}\Big)^q \cdot \mathbb{E}\Big(\sup_{u\in T_s}\langle g, u\rangle\Big)^q, \quad g\in \mathcal{N}(0,I_p).
\end{align}

We first construct an upper bound for $\E\big( \sup_{u\in T_s} \langle g, u\rangle \big)^{q} $. Using Minkowski's inequality,
    \begin{equation}\label{eq::moments-of-gaussian-complexity}
        \bigg[ \E \Big( \sup_{u\in T_s} ~ \langle g, u\rangle \Big)^{q} \bigg]^{1/q} \leq \bigg( \E \Big| \sup_{u\in T_s} \langle g, u\rangle - \E \sup_{u\in T_s} ~ \langle g, u \rangle \Big|^{q} \bigg)^{1/q} + \E \sup_{u\in T_s}~ \langle g, u \rangle .
    \end{equation}
    The second term above is the Gaussian complexity of $T_s$, and it has a sharp upper bound (e.g. Exercise 5.7 in \cite{wainwright2019high}), 
    \begin{equation}\label{eq::gaussian-complexity}
        \E \sup_{u\in T_s} \langle g, u\rangle \leq C\sqrt{ s \log (ep/s)},
    \end{equation}
for some absolute constant $C>0$.
    To bound the first term in \eqref{eq::moments-of-gaussian-complexity}, let $F(g):= \sup_{u\in T_s}\langle g, u\rangle$. Then, it is clear that $F(\cdot)$ is a $1$-Lipschitz function. Using the concentration of Lipschitz function of Gaussians (e.g., Theorem 2.26 in \cite{wainwright2019high}), we obtain
    \begin{eqnarray}
        && \E ~\Big| \sup_{u\in T(s)}  \langle g, u\rangle - \E \sup_{u\in T(s)}  \langle g, u \rangle \Big|^{q} = \int_{0}^{\infty} qt^{q-1}P\Big( |F(g)-\E F(g)| > t \Big) dt \nonumber \\
        &\leq& \int_{0}^{\infty} 2 qt^{q-1} e^{-\frac{t^{2}}{2}} d t  = 2^{\frac{q}{2}}q\Gamma(\frac{q}{2}), \label{eq::concentration-of-gaussian-complexity}
    \end{eqnarray}
    where $\Gamma(\cdot)$ is the Gamma function. Putting together \eqref{eq::moments-of-gaussian-complexity}-\eqref{eq::concentration-of-gaussian-complexity} gives us
    \begin{align}
    \label{term2:nonasym}
     \E \Big( \sup_{u\in T_s} ~ \langle g, u\rangle \Big)^{q}\leq C_q \cdot (\sqrt{s\log(ep/s)})^q,
    \end{align}
    for some constant $C_q>0$ only depending on $q$. Finally, note that $\frac{\|z\|_2}{\sqrt{n}}$, as a function of the standard Gaussian $z/\sigma$, is a $(\frac{\sigma}{\sqrt{n}})$-Lipschitz function. Hence, we can use similar arguments to derive the bound for $\mathbb{E}\Big(\frac{\|z\|_2}{\sqrt{n}}\Big)^q$: there exists some constant $\tilde{C}_q>0$ only depending on $q$ such that
\begin{align}
\label{term1:nonasym}
\mathbb{E}\Big(\frac{\|z\|_2}{\sqrt{n}}\Big)^q\leq \tilde{C}_q \sigma^q.
\end{align}
Combining \eqref{basic:form:eq}, \eqref{term2:nonasym} and \eqref{term1:nonasym} finishes the proof.

\end{proof}

\subsection{Proof of Proposition \ref{prop:lasso:denoiser}}\label{ssec:proof:prop5}

 We first state some useful inequalities in the lemma below.
\begin{lemma}
\label{useful:basic:eq}
Under $y=X\beta+z$, the Lasso estimator $\betaL$ in \eqref{eq::lasso-estimator0} satisfies 
\begin{itemize}
\item[(i)] $\frac{1}{n}\|X(\betaL-\beta)\|_2^2\leq \frac{1}{n}z^TX(\betaL-\beta)-\lambda\|\betaL_{S^c}-\beta_{S^c}\|_1-\lambda(\betaL_S-\beta_S)^T{\rm sign}(\beta_S)$, where $S={\rm supp}(\beta)$.
\item[(ii)] $\frac{1}{n}\|X(\betaL-\beta)\|_2^2\leq \|\betaL-\beta\|_2\sqrt{\sum_{j\in S}(g_j-\lambda {\rm sign}(\beta_j))^2+\sum_{j\in S^c}(|g_j|-\lambda)_+^2}$, where $g_j=\frac{1}{n}X_j^Tz$.
\end{itemize}
\end{lemma}
\begin{proof}
Applying the KKT condition to \eqref{eq::lasso-estimator0} gives
\begin{align*}
\frac{1}{n}X^T(y-X\betaL)=\lambda d, {\rm~where~}d_j={\rm sign}(\betaL_j) {\rm~if~} \betaL_j\neq 0 {\rm~and~} |d_j|\leq 1 {\rm~if~}\betaL_j= 0.
\end{align*}
Plugging in $y=X\beta+z$ and multiplying both sides above by $\beta-\betaL$ yields
\begin{align*}
\frac{1}{n}\|X(\betaL-\beta)\|_2^2&=\frac{1}{n}(\betaL-\beta)^TX^Tz-\lambda(\betaL-\beta)^Td \\
&=\frac{1}{n}(\betaL-\beta)^TX^Tz-\lambda(\betaL_{S^c}-\beta_{S^c})^Td_{S^c}-\lambda(\betaL_S-\beta_S)^Td_S.
\end{align*}
Part (i) is proved by further noting
\begin{align*}
&(\betaL_{S^c}-\beta_{S^c})^Td_{S^c}=(\betaL_{S^c})^Td_{S^c}=\|\betaL_{S^c}\|_1=\|\betaL_{S^c}-\beta_{S^c}\|_1, \\
& (\betaL_S)^Td_S=\|\betaL_S\|_1\geq (\betaL_S)^T{\rm sign}(\beta_S),~\beta_S^Td_S\leq \|\beta_S\|_1=(\beta_S)^T{\rm sign}(\beta_S).
\end{align*}
Regarding Part (ii), let $u=\betaL-\beta$. The bound from Part (i) can be rewritten as
\begin{align*}
\frac{1}{n}\|X(\betaL-\beta)\|_2^2&\leq \sum_{j\in S}\big(g_j-\lambda{\rm sign}(\beta_j)\big)u_j +\sum_{j\in S^c}(g_ju_j-\lambda|u_j|) \\
&\leq \sum_{j\in S}\big(g_j-\lambda{\rm sign}(\beta_j)\big)u_j+\sum_{j\in S^c}(|g_j|-\lambda)_+|u_j| \\
&\leq \|u\|_2\cdot \sqrt{\sum_{j\in S}(g_j-\lambda {\rm sign}(\beta_j))^2+\sum_{j\in S^c}(|g_j|-\lambda)_+^2},
\end{align*}
where the last inequality is by Cauchy–Schwarz inequality.
\end{proof}

For a given $u=(u_{1},\ldots,u_{p})\in\R^{p}$, let $|u|_{(1)}\geq |u|_{(2)} \geq \cdots \geq |u|_{(p)}$ denote the order statistics of $\big(|u_{1}|, \ldots, |u_{p}|\big)$, and define
    \begin{align}
       & H(u):= \sigma(1+\delta_2) \Bigg(\sum_{j=1}^{k} |u|_{(j)} 4\sqrt{\frac{\log(2p/j)}{n}} + (1+\delta_1) \sum_{j=k+1}^{p} |u|_{(j)}\sqrt{\frac{2\log(p/k)}{n}}\Bigg), \label{eq::H-G-function1}\\
       & G(u):= \sigma(1+\delta_2) \delta_2^{-1} \frac{\sqrt{2\log(1/\delta_3)}}{n(1+\delta_0)} \|Xu\|_{2}, \label{eq::H-G-function2}
    \end{align}
where $\delta_0, \delta_1,\delta_2,\delta_3 \in (0,1)$ are some constants. The following lemma shows the importance of the two terms defined above. 

\begin{lemma}[Bound on the stochastic error]\label{lem::bound-on-stochastic-error}
    Let $z\sim \mathcal{N}(0,\sigma^2I_{n})$ and $X\in \R^{n\times p}$ be a fixed matrix such that $\max_{j\in [p]} \|X_{j}\|_{2} \leq (1+\delta_{0})\sqrt{n}$. For any given constants $\delta_0,\delta_1,\delta_2,\delta_3\in (0,1)$, there exist a constant $C_{\delta_1}>0$ only depending on $\delta_1$ such that the event
\begin{equation}
    \Big \{\frac{1}{n} z^{T}Xu \leq (1+\delta_{0})\cdot \max\big ( H(u), G(u) \big), ~ \forall u \in \R^{p} \Big\} \label{eq::stochastic-error-event}
\end{equation}
holds with probability at least $1-\delta_3$, as long as $p/k \geq C_{\delta_1}$.
\end{lemma}
\begin{proof}
The lemma is a modified version of Theorem 4.1 in \cite{bellec2018slope}, tailored for the smaller tuning parameter value $\lambda_{\varepsilon}$. The proof is similar, hence we do not repeat all the details and only mention the major difference in bounding $\frac{1}{n}z^TXu$. Let $g_j=\frac{z^TX_j}{\sigma\sqrt{n}(1+\delta_0)}, j=1,2,\ldots, p$. We have
\begin{align}
&\quad ~\frac{1}{n}z^TXu \leq (1+\delta_0)\sigma \sum_{j=1}^p \frac{1}{\sqrt{n}}|g|_{(j)}|u|_{(j)} \nonumber \\
& \leq (1+\delta_0)\sigma \sum_{j=1}^k\Big(|u|_{(j)} 4\sqrt{\frac{\log(2p/j)}{n}}\Big)\cdot \Big(\frac{|g|_{(j)}}{4\sqrt{\log(2p/j)}}\Big) \nonumber \\
&~~+(1+\delta_0)\sigma\sum_{j=k+1}^p\Big((1+\delta_1)|u|_{(j)}\sqrt{\frac{2\log(p/k)}{n}}\Big)\cdot\Big(\frac{|g|_{(j)}}{(1+\delta_1)\sqrt{2\log(p/k)}}\Big) \nonumber \\
&\leq \frac{(1+\delta_0)H(u)}{1+\delta_2} \cdot \Big( \max_{1\leq j\leq k} \frac{|g|_{(j)}}{4\sqrt{\log(2p/j)}} \Big) \vee \frac{|g|_{(k+1)}}{(1+\delta_1)\sqrt{2\log(p/k)}}. \label{med:for:use}
\end{align}
Define 
\[
\mathcal{T}=\Big\{u\in \mathbb{R}^p: \max\big(H(u),G(u)\big)\leq \frac{1+\delta_2}{1+\delta_0} \Big\}.
\]
The rest is to bound $\max_{u\in\mathcal{T}}\frac{1}{n}z^TXu$, using concentration of Gaussian measure around the median and using \eqref{med:for:use} together with Lemma \ref{lem::median-inequality} to bound the median. The detail is similar to that in \cite{bellec2018slope} and is hence skipped.
\end{proof}

We use the basic inequality from Part (i) of Lemma \ref{useful:basic:eq} as the starting point, and apply Lemma \ref{lem::bound-on-stochastic-error} to control the noise-related term. Along the way, the inequality from Part (ii) of Lemma \ref{useful:basic:eq} is employed to obtain the sharp constant. We provide the details in the following lemma and its proof.

{\color{black}
\begin{lemma}\label{lem::lasso-l2-bound}
Assume $p\geq 2k$. For any $\delta_0,\delta_1,\delta_2,\delta_3\in (0,1)$ satisfying $(1+\varepsilon)-(1+\delta_0)(1+\delta_1)(1+\delta_2)>0$, define
\begin{align}
\label{c0:exact:form}
c_0=\frac{4\sqrt{2}(1+\delta_0)(1+\delta_2)+1+\varepsilon}{(1+\varepsilon)-(1+\delta_0)(1+\delta_1)(1+\delta_2)}.
\end{align}
Consider model \eqref{model::gaussian-model} with any fixed design $X \in \mathcal{A}(\delta_0, c_0, k)$ in \eqref{eq::event-A}. Then, on the event \eqref{eq::stochastic-error-event}, the Lasso estimator $\betaL$ in \eqref{eq::lasso-estimator0} with tuning parameter $\lambda_{\varepsilon}$ satisfies 
\begin{align*}
\|\betaL-\beta\|_2^2\leq &~ \frac{1}{(1-\delta_0)^4}\Bigg(\sum_{j\in S}(g_j-\lambda_{\varepsilon} {\rm sign}(\beta_j))^2+\sum_{j\in S^c}(|g_j|-\lambda_{\varepsilon})_+^2\Bigg)\\
&+ C^2(\delta_0,\delta_2, \varepsilon) \frac{\sigma^2(\log(1/\delta_3))^2}{nk\log(p/k)},
\end{align*}
where  $g_j=\frac{1}{n}X_j^Tz, S={\rm supp}(\beta)$ and
\begin{align*}
C(\delta_0,\delta_2, \varepsilon):=\frac{4\sqrt{2}(1+\delta_0)(1+\delta_2)+1+\varepsilon}{16\sqrt{2}(1+\delta_0)^2\delta_2^2}.
\end{align*}
\end{lemma}
\begin{proof}[Proof of Lemma \ref{lem::lasso-l2-bound}]
We start with the basic inequality from Part (i) of Lemma \ref{useful:basic:eq},
    \begin{align}
       \frac{1}{n}\|X(\betaL-\beta)\|_2^2\leq \frac{1}{n}z^TX(\betaL-\beta)-\lambda_{\varepsilon}\|\betaL_{S^c}-\beta_{S^c}\|_1-\lambda_{\varepsilon}(\betaL_S-\beta_S)^T{\rm sign}(\beta_S). \label{eq::triangle-star}
    \end{align}
Denoting $u=\betaL-\beta$ and using the basic result $(\betaL_S-\beta_S)^T{\rm sign}(\beta_S)\geq -\|\betaL_S-\beta_S\|_1$, we can continue from \eqref{eq::triangle-star} to obtain that on the event \eqref{eq::stochastic-error-event},  
\begin{align*}
\frac{1}{n}\|Xu\|_2^2 \leq (1+\delta_0)\max(H(u),G(u))+\lambda_{\varepsilon}(\|u_{S}\|_{1} - \|u_{S^{c}}\|_{1}).
\end{align*}
To have an upper bound for $H(u)$, define
    \begin{equation}
    \label{tH:def}
        \Tilde{H}(u) := \sigma(1+\delta_2) \bigg(8 \|u\|_{2} \sqrt{\frac{k\log(p/k)}{n}} + (1+\delta_1) \sqrt{\frac{2\log(p/k)}{n}}\|u_{S^c}\|_1\bigg).
    \end{equation}
    Using the Cauchy-Schwarz inequality, 
    \begin{align*}
    H(u)&\leq \sigma(1+\delta_2)\bigg(4 \|u\|_{2} \sqrt{\sum_{j=1}^{k}\frac{\log(2p/j)}{n}} + (1+\delta_1) \sqrt{\frac{2\log(p/k)}{n}}\sum_{j=k+1}^{p} |u|_{(j)} \bigg) \\
    & \leq \sigma(1+\delta_2)\bigg(4 \|u\|_{2} \sqrt{\frac{k\log(2ep/k)}{n}} + (1+\delta_1) \sqrt{\frac{2\log(p/k)}{n}}\sum_{j=k+1}^{p} |u|_{(j)} \bigg)\\
    & \leq \Tilde{H}(u),
    \end{align*}
    where the second inequality holds because $\sum_{j=1}^{k}\log(2p/j) = k\log(2p) - \log(k!) \leq k\log(2ep/k)$ by Stirling's formula, and the third one is due to the fact that $\log(2ep/k)\leq 4\log(p/k)$ (as $p\geq 2k$) and $\sum_{j=k+1}^{p} |u|_{(j)}\leq \|u_{S^c}\|_1$. Hence, we can further proceed to have that on the event \eqref{eq::stochastic-error-event},
    \begin{align}
    \label{base:inq}
    \frac{1}{n}\|Xu\|_2^2 \leq (1+\delta_0)\max(\tilde{H}(u),G(u))+\lambda_{\varepsilon}(\|u_{S}\|_{1} - \|u_{S^{c}}\|_{1}).
    \end{align}

    Based on \eqref{base:inq}, we will derive the bound for $\|u\|_2$ in two different cases:
    \begin{itemize}
        \item If $G(u)>\tilde{H}(u)$, then by comparing the two expressions given in \eqref{eq::H-G-function2} and \eqref{tH:def} we have
        \begin{equation}
            \|u\|_{2} \leq \frac{\sqrt{2\log(1/\delta_3)}}{8(1+\delta_0)\delta_2\sqrt{nk\log(p/k)}}\|Xu\|_2.\label{eq::u2-bounded-by-xu2}
        \end{equation}
        This together with \eqref{base:inq} gives us
        \begin{align*}
            \frac{1}{n}\|Xu\|^2_2 &\leq (1+\delta_0)G(u)+\sqrt{k}\lambda_{\varepsilon}\|u\|_2 \\
            &\leq \sigma(1+\delta_2)\frac{\sqrt{2\log(1/\delta_3)}}{n\delta_2}\|Xu\|_2+\sigma(1+\varepsilon)\sqrt{\frac{2k\log(p/k)}{n}}\|u\|_2 \\
           & \leq \Big(1+\delta_2+\frac{1+\varepsilon}{4\sqrt{2}(1+\delta_0)}\Big)\cdot \frac{\sigma\sqrt{2\log(1/\delta_3)}}{n\delta_2}\|Xu\|_2.
        \end{align*}
        Solving the above for $\|Xu\|_2$ and then plugging it into \eqref{eq::u2-bounded-by-xu2}, we obtain
        \begin{align}
        \label{sharp:bound:first}
        \|u\|_2\leq \textcolor{black}{C(\delta_0,\delta_2,\varepsilon)}\frac{\sigma\log(1/\delta_3)}{\sqrt{nk\log(p/k)}}.
        \end{align}

        \item If $G(u)\leq \tilde{H}(u)$, using $\|u_S\|_1\leq \sqrt{k}\|u\|_2$ in \eqref{base:inq} yields
        \begin{align}
          0\leq \frac{1}{n}\|Xu\|_2^2 \leq &~\Big(8(1+\delta_0)(1+\delta_2)+\sqrt{2}(1+\varepsilon)\Big)\sigma\sqrt{\frac{k\log(p/k)}{n}}\|u\|_2 \nonumber \\
          &\hspace{-0.4cm}-\Big((1+\varepsilon)-(1+\delta_0)(1+\delta_1)(1+\delta_2)\Big)\sigma\sqrt{\frac{2\log(p/k)}{n}}\|u_{S^c}\|_1, \label{basic:eq:case2}
        \end{align}
        which implies $\|u_{S^c}\|_1 \leq c_0\sqrt{k}\|u\|_2$ with $c_0$ defined in \eqref{c0:exact:form}. This further shows that $\|u\|_1=\|u_S\|_1+\|u_{S^c}\|_1\leq (1+c_0)\sqrt{k}\|u\|_2$. Therefore, applying $\theta(k,c_0)$ in \eqref{sre:condef} and the condition $X\in \mathcal{A}(\delta_0,c_0,k)$, we obtain 
        \[
        (1-\delta_0)^2\|u\|^2_2\leq \frac{1}{n}\|Xu\|_2^2.
        \]
        \textcolor{black}{
       Combining this result with Part (ii) of Lemma \ref{useful:basic:eq} gives us
       \begin{align*}
       (1-\delta_0)^2\|u\|_2^2\leq \|u\|_2\cdot \sqrt{\sum_{j\in S}(g_j-\lambda_{\varepsilon} {\rm sign}(\beta_j))^2+\sum_{j\in S^c}(|g_j|-\lambda_{\varepsilon})_+^2}
       \end{align*}
       leading to the bound
       \begin{align}
       \label{sharp:bound:second}
      \|u\|_2 \leq \frac{1}{(1-\delta_0)^2}\sqrt{\sum_{j\in S}(g_j-\lambda_{\varepsilon} {\rm sign}(\beta_j))^2+\sum_{j\in S^c}(|g_j|-\lambda_{\varepsilon})_+^2}.
       \end{align}
       }
       \end{itemize}

Putting together \eqref{sharp:bound:first} and \eqref{sharp:bound:second} completes the proof.

\end{proof}
}

According to Lemmas \ref{lem::bound-on-stochastic-error} and \ref{lem::lasso-l2-bound}, we can conclude: for any given constants $\delta_0,\delta_1,\delta_2,\delta_3\in (0,1)$ satisfying $(1+\varepsilon)-(1+\delta_0)(1+\delta_1)(1+\delta_2)>0$ and any fixed design $X \in \mathcal{A}:=\mathcal{A}(\delta_0, c_0, k)$ with $c_0$ in \eqref{c0:exact:form}, the Lasso estimator $\betaL$ in \eqref{eq::lasso-estimator0} with tuning parameter $\lambda_{\varepsilon}$ satisfies 
\begin{align*}
\|\betaL-\beta\|_2^2\leq &~ \underbrace{\frac{1}{(1-\delta_0)^4}\Bigg(\sum_{j\in S}(g_j-\lambda_{\varepsilon} {\rm sign}(\beta_j))^2+\sum_{j\in S^c}(|g_j|-\lambda_{\varepsilon})_+^2\Bigg)}_{\hat{\Delta}}\\
&+ \underbrace{C^2(\delta_0,\delta_2, \varepsilon) \frac{\sigma^2}{nk\log(p/k)}}_{b}(\log(1/\delta_3))^2,
\end{align*}
with probability at least $1-\delta_3$, provided $p/k>\max(2,C_{\delta_1})$. Such a result implies
\begin{align*}
\mathbb{P}\Big(\frac{1}{b}(\|\betaL-\beta\|_2^2-\hat{\Delta})\mathbbm{1}_{\mathcal{A}}>t\Big)\leq e^{-\sqrt{t}}, ~\forall t>0.
\end{align*}
We can thus continue to obtain
\begin{align}
\label{from:high:exp}
\mathbb{E}\frac{1}{b}(\|\betaL-\beta\|_2^2-\hat{\Delta})\mathbbm{1}_{\mathcal{A}}&\leq \int_0^{\infty}\mathbb{P}\Big(\frac{1}{b}(\|\betaL-\beta\|_2^2-\hat{\Delta})\mathbbm{1}_{\mathcal{A}}>t\Big)\nonumber \\
&\leq \int_0^{\infty}e^{-\sqrt{t}}\leq \bar{C},
\end{align}
where the first inequality is due to the integral identity $\mathbb{E}(W)=\int_0^{\infty}\mathbb{P}(W>t)dt-\int_{-\infty}^0\mathbb{P}(W<t)dt, \forall W\in \mathbb{R}$, and $\bar{C}$ is an absolute constant. We can write out \eqref{from:high:exp} more explicitly as
\begin{align}
\label{exp:res:end}
&~\mathbb{E}\|\betaL-\beta\|_2^2\mathbbm{1}_{\mathcal{A}}\nonumber \\
\leq &~\frac{1}{(1-\delta_0)^4}\Bigg(\sum_{j\in S}\mathbb{E}(g_j-\lambda_{\varepsilon} {\rm sign}(\beta_j))^2\mathbbm{1}_{\mathcal{A}}+\sum_{j\in S^c}\mathbb{E}(|g_j|-\lambda_{\varepsilon})_+^2\mathbbm{1}_{\mathcal{A}}\Bigg)\nonumber \\
&~+\bar{C}C^2(\delta_0,\delta_2, \varepsilon) \frac{\sigma^2}{nk\log(p/k)}.
\end{align}
Since $g_j|X\sim \mathcal{N}(0,\sigma^2n^{-2}\|X_j\|^2_2)$, it holds that 
\[
\mathbb{E}(g_j-\lambda_{\varepsilon} {\rm sign}(\beta_j))^2\mathbbm{1}_{\mathcal{A}}=\mathbb{E}(g_j-\lambda_{\varepsilon})^2\mathbbm{1}_{\mathcal{A}}.
\]
Plugging the above into \eqref{exp:res:end} and take $\sup_{\beta\in \Theta_k}$ on both sides of the inequality leads to
\begin{align*}
&~\sup_{\beta\in \Theta_k}\mathbb{E}\|\betaL-\beta\|_2^2\mathbbm{1}_{\mathcal{A}} \\
\leq &~\frac{1}{(1-\delta_0)^4}\sup_{|S|\leq k}\Bigg\{\sum_{j\in S}\mathbb{E}(g_j-\lambda_{\varepsilon})^2\mathbbm{1}_{\mathcal{A}}+\sum_{j\in S^c}\mathbb{E}(|g_j|-\lambda_{\varepsilon})_+^2\mathbbm{1}_{\mathcal{A}}\Bigg\} \\
&~+\bar{C}C^2(\delta_0,\delta_2, \varepsilon) \frac{\sigma^2}{nk\log(p/k)}.
\end{align*}
We choose
\[
\delta_0=\delta_1=\delta_2=\Big(1+\frac{\varepsilon}{2}\Big)^{\frac{1}{3}}-1.
\]
As a result, $(1+\delta_0)(1+\delta_1)(1+\delta_2)=1+\frac{\varepsilon}{2}<1+\varepsilon$, and the constant $c_0$ of \eqref{c0:exact:form} becomes exactly the one specified in \eqref{cons:value} (so is $\delta_0$). The proof of Proposition \ref{prop:lasso:denoiser} is completed by setting $C_{\varepsilon}:=\bar{C}C^2(\delta_0,\delta_2, \varepsilon)$ and $c_{\varepsilon}:=\max(2,C_{\delta_1})$.

\printbibliography

\end{document}